\documentclass{amsart}
\usepackage{graphicx}
\usepackage{amsfonts,amssymb,amsmath,amsthm,mathrsfs, verbatim,enumitem,hyperref, mathtools}
\usepackage{tikz}
\usetikzlibrary{positioning}
\usetikzlibrary{arrows}
\usepackage{comment}
\usepackage{caption}
\usepackage{subcaption}
\usepackage{color}
\usepackage[all]{xy}
\newcommand{\kk}{k}
\newcommand{\tors}{\operatorname{tors}}
\newcommand{\supp}{\operatorname{supp}}
\newcommand{\suppo}{\supp^{\circ}}
\newcommand{\image}{\operatorname{Im}}
\newcommand{\torf}{\operatorname{torf}}
\newcommand{\op}{\operatorname{op}}

\newcommand{\calE}{\mathcal{E}}
\newcommand{\calF}{\mathcal{F}}
\newcommand{\calG}{\mathcal{G}}

\newcommand{\calI}{\mathcal{I}}

\newcommand{\calS}{\mathcal{S}}
\newcommand{\calT}{\mathcal{T}}

\newcommand{\calY}{\mathcal{Y}}

\newcommand{\onto}{\twoheadrightarrow}
\newcommand{\add}{\operatorname{add}}

\newcommand{\filt}{\mathscr{F}ilt}

\newcommand{\Hom}{\operatorname{Hom}}

\newcommand{\End}{\operatorname{End}}

\newcommand{\ME}{\operatorname{ME}}%For minimal extending

\newcommand{\ind}{\operatorname{ind}}
\newcommand{\cov}{\operatorname{cov}}
\newcommand{\covers}{{\,\,\,\cdot\!\!\!\! >\,\,}}
\newcommand{\covered}{{\,\,<\!\!\!\!\cdot\,\,\,}}

\newcommand{\covdown}{\cov_{\downarrow}}
\newcommand{\covup}{\cov^{\uparrow}}

\newcommand{\inv}{\operatorname{inv}}

\newcommand{\join}{\vee}
\renewcommand{\Join}{\bigvee}
\newcommand{\Meet}{\bigwedge}

\newcommand{\refines}{{\,\underline{\ll}\,}}
\newcommand{\tran}{\operatorname{tran}}
\newcommand{\Gen}{\operatorname{Gen}}

\newcommand{\ftors}{\operatorname{f-tors}}
\renewcommand{\1}{\hat{1}}

\newcommand{\newword}[1]{\textbf{\emph{#1}}}
\newcommand{\module}{\operatorname{mod}}
\newtheorem{proposition}{Proposition}[subsection]
\newtheorem{theorem}[proposition]{Theorem}
\newtheorem{corollary}[proposition]{Corollary}
\newtheorem{lemma}[proposition]{Lemma}

\newtheorem{example}[proposition]{Example}
\newtheorem{definition}[proposition]{Definition}

\newtheorem{remark}[proposition]{Remark}

\numberwithin{equation}{section}

\newcommand{\margincolor}{red}

\newcounter{margincounter}
\newcommand{\marginnum}{\textcolor{\margincolor}{\begin{picture}(0,0)\put(5,3){\circle{13}}\end{picture}\arabic{margincounter}}}

\newcommand{\margin}[1]
{\marginnum\marginpar{\textcolor{\margincolor}{\arabic{margincounter}.}\,\,\tiny
    #1}\addtocounter{margincounter}{1}}

\newcounter{amargincounter}
\newcounter{smargincounter}
 \renewcommand{\margin}[1]{}
\begin{document}
\title{Minimal inclusions of torsion classes}
\author{Emily Barnard, Andrew T. Carroll, and Shijie Zhu}
\maketitle

\begin{abstract}
Let $\Lambda$ be a finite-dimensional associative algebra.
The torsion classes of $\module\Lambda$ form a lattice under containment, denoted by $\tors \Lambda$.
In this paper, we characterize the cover relations in $\tors \Lambda$ by certain indecomposable modules.
We consider three applications:
First, we show that the completely join-irreducible torsion classes
(torsion classes which cover precisely one element) are in bijection
with bricks.
Second, we characterize faces of the canonical join complex of
$\tors \Lambda$ in terms of representation theory.
Finally, we show that, in general, the algebra $\Lambda$ is not characterized by its lattice $\tors \Lambda$.
In particular, we study the torsion theory of a quotient of the preprojective algebra of type $A_n$.
We show that its torsion class lattice is isomorphic to the weak order on~$A_n$.
\end{abstract}

%added a table of contents
\setcounter{tocdepth}{2}
\tableofcontents

\section{Introduction}
Let $\Lambda$ be a
finite-dimensional associative algebra over a field~$k$, and write
$\tors \Lambda$ for the set of torsion classes of finitely generated
modules over $\Lambda$, partially ordered by containment.
The poset $\tors \Lambda$ is a complete lattice in which the \newword{meet} (or greatest lower bound) $\Meet \{\calT, \calT'\}$ coincides with the intersection $\calT \cap \calT'$, and the \newword{join} (or smallest upper bound) $\Join \{\calT, \calT'\}$ coincides with the iterative extension closure of the union $\calT\cup \calT'$.
In this paper, we study the cover relations $\calT'\covers \calT$ in $\tors \Lambda$.
Recall that a torsion class $\calT'$ \newword{covers} $\calT$ if $\calT\subsetneq \calT'$ and for each $\calY\in \tors\Lambda$, if $\calT\subsetneq \calY \subseteq \calT'$ then $\calY= \calT'$.

In \cite{AIR}, the authors describe the lattice of functorially-finite torsion classes by way of $\tau$-tilting pairs.
They show the existence of a unique module which encodes each cover relation as follows:
When $\calT$ and $\calT'$ are functorially finite torsion classes, with $\calT'\covers \calT$,
there exists a unique module $M$ with the property that $\calT'$ is the closure of $\add(\calT \cup\{M\})$ under taking quotients.

In our complimentary approach, we show that for each cover relation $\calT'\covers \calT$ in $\tors\Lambda$, there exists a unique ``minimal'' module $M$ with the property that $\calT'$ is the closure of $\calT\cup \{M\}$ under taking \textit{iterative extensions}.
Below, we make the notation of ``minimal'' precise. 

\begin{definition}\label{def:min-ext-module}
A module $M$ is a \newword{minimal extending module for $\calT$} if it satisfies the following three properties:
\begin{enumerate}[label={(P\arabic*)}]
 \item Every proper factor of \label{propertyone}
   $M$ is in $\calT$;
  \item If $0\rightarrow M\rightarrow X
    \rightarrow T\rightarrow 0$ is a non-split exact sequence with
    $T\in \calT$, then $X\in \calT$; \label{propertytwo}
  \item $\Hom(\calT, M)=0$. \label{propertythree}
  \end{enumerate}
\end{definition}

The following two remarks will be useful. 
First, Property \ref{propertyone} implies that any minimal extending
module is indecomposable. Indeed, direct summands are proper factors, and
torsion classes are closed under direct sums. 
Second, assuming Property \ref{propertyone}, Property
\ref{propertythree} is equivalent to the fact that $M\notin \calT$. In
particular, if there is a non-trivial homomorphism from a module in
$\calT$ to $M$, then both the image and cokernel are in $\calT$, and
$M$ is an extension of the cokernel by the image.

In the statement of the following theorem, and throughout the paper we have the following notation:
We write $[M]$ for the isoclass of the module $M$; $\ME(\calT$) for the set  of isoclasses~$[M]$ such that $M$ is a minimal extending module for $\calT$; and $\filt(\calT \cup \{M\})$ for the iterative extension closure of $\calT \cup \{M\}$.

\begin{theorem}\label{main: covers}
Suppose $\calT$ be a torsion class over $\Lambda$.
Then the map $$\eta_\calT: [M]\mapsto \filt(\calT\cup \{M\})$$ is a bijection from the set $\ME(\calT)$ to the set of $\calT' \in \tors \Lambda$ such that  $\calT \covered \calT'$.
\end{theorem}

Recall that a module $M$ is called a \newword{brick} if the
endomorphism ring of $M$ is a division ring. (That is, the non-trivial
endomorphisms are invertible.)

\begin{theorem}\label{main: cor schur}
Let $\Lambda$ be a finite-dimensional associative algebra and $M\in\module\Lambda$.
Then $M$ is a minimal extending module for some torsion class  if and
only if $M$ is a brick.
\end{theorem}

As an immediate consequence of Theorem~\ref{main: cor schur} we
obtain a labeling of cover relations for $\tors
\Lambda$ (when they exist) by bricks. We illustrate in
Example~\ref{kronecker} that in $\tors
\Lambda$ there exists pairs $\calT'> \calT$ such that there is no
torsion class~$\calY$ satisfying $\calT' \covers \calY \ge \calT$.

This paper fits into a larger body of research which studies the combinatorial structure of the lattice of (functorially finite) torsion classes.
Connections between the combinatorics of a finite simply laced Weyl group $W$ and the corresponding preprojective algebra $\Pi W$ are of particular interest.
In \cite{Miz}, Mizuno showed the lattice of functorially finite torsion classes $\ftors \Pi W$ is isomorphic to the weak order on $W$.
Building on this work, the authors of \cite{IRRT} have shown that the lattice of (functorially finite) torsion classes of quotients of $\Pi W$ are \textit{lattice quotients} of the weak order on~$W$.
They also obtain an analogous labeling of $\ftors \Pi W$ by certain modules called \textit{layer modules}.
(See \cite[Theorem~1.3]{IRRT}.)
In a related direction, the authors of \cite{G-M} study certain lattice properties of $\ftors \Lambda$ when $\Lambda$ arises from a quiver that is mutation-equivalent to a path quiver or oriented cycle, including a certain minimal ``factorization'' called the \textit{canonical join representation}.

Inspired by these results, we will give three applications of Theorem~\ref{main: covers} and~\ref{main: cor schur}.
Before describing them, we recall some terminology in lattice theory.
In a (not necessarily finite) lattice $L$, a \newword{join representation} for an element $w\in L$ is an expression $\Join A = w$, where $A$ is a (possibly infinite) subset of $L$.
An element $w$ is \newword{join-irreducible}, if $w\in A$ for any join representation $w=\Join A$, where $A$ is a finite set.
An element $w$ is \newword{completely join-irreducible}, if $w\in A$ for any join representation $w=\Join A$, where $A$ is any subset of $L$.
Equivalently, $w$ is completely join-irreducible if and only if $w$ covers only one element $v$ and any element $u<w\in L$ satisfies $u\leq v$.
\begin{example}\label{natural numbers}
\normalfont
Let $L$ be the lattice of $\mathbb{N}\cup \{\infty\}$ with natural partial order. Then $\infty$ is join-irreducible but not completely join-irreducible.
\end{example}

For our first application, we have the following theorem; see \cite[Lemma~8.2]{G-M} and \cite[Theorems~1.1 and 1.2]{IRRT} for analogous results.
(In the statement below $\Gen(M)$ is the closure of $M$ under taking factors and direct sums.)

\begin{theorem}\label{schur modules to join-irreducibles}
Suppose that $M$ is an indecomposable $\Lambda$ module.
Then the map $\zeta:[M]\mapsto \filt(\Gen(M))$ is a bijection from the
set of isoclasses of bricks over $\Lambda$ to the set of of completely join-irreducible elements in $(\tors \Lambda,  \subseteq)$.
\end{theorem}

For our second application, we consider the canonical join representation in $\tors \Lambda$.
A join representation $\Join A$ is \newword{irredundant} if $\Join A' <\Join A$ for each subset $A'\subseteq A$.
Informally, the \newword{canonical join representation of $w$} is the unique lowest irredundant join representation $\Join A$ for $w$, when such an expression exists.
In this case, we also say that the set $A$ is a canonical join representation.
(We make the notion ``lowest'' precise in Section~\ref{canonical join complex}.)
In particular, each element $a\in A$ is join-irreducible.
\begin{example}\label{Tamari example}
\normalfont
Consider the top element $\1$ in the pentagon lattice $N_5$ shown on the left in Figure~\ref{pentagon}.
Its irredundant join representations are $\Join \1$, $\Join \{x,z\}$ and $\Join \{x,y\}$.
Observe that $\Join \{x,y\}$ is ``lower'' than $\Join\{x,z\}$ because $x\le z$ (and $y\le y$).
Indeed, $\Join \{x,y\}$ is the canonical join representation of $\1$.
\end{example}

It is natural to ask which collections $A$ of join-irreducible elements in~$L$ satisfy $\Join A$ is a canonical join representation.
The collection of such subsets, which we denote by $\Gamma(L)$, has the structure of an abstract simplicial complex whose vertex set is the set of join-irreducible elements in $L$.
(See, for instance, \cite[Proposition~2.2]{arcs}, for the case when $L$ is finite.)
We call $\Gamma(L)$ the \newword{canonical join complex} of $L$.
\begin{figure}[h]
  \centering
   \scalebox{.8}{ \includegraphics{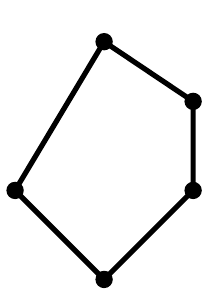}
   \begin{picture}(0,50)(0,50)
   \put(-2,82){$y$}
   \put(-2,106){$z$}
   \put(-68.5,82){$x$}
   \put(-36, 127){$\1$}
   \end{picture}
   \qquad \qquad
      \scalebox{.8}{ \includegraphics{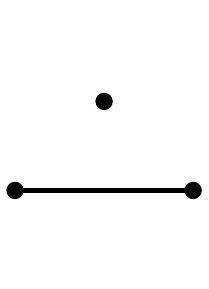}
   \begin{picture}(0,50)(0,50)
   \put(-9.5,65){$y$}
   \put(-27,95){$z$}
   \put(-60.5,65){$x$}
   \end{picture}}}
     \caption{A lattice $N_5$ and its canonical join complex $\Gamma(N_5)$.}
        \label{pentagon}
\end{figure}
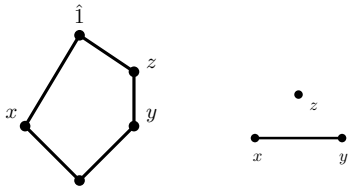

\begin{example}\label{ex: canonical join complex}
\normalfont
Consider the pentagon lattice $N_5$ shown on the left in Figure~\ref{pentagon}.
The canonical join complex $\Gamma(N_5)$ appears on the right.
%The vertices of $\Gamma(N_5)$ are precisely the join-irreducible elements in $N_5$.
%The canonical join representation of $\1$ is $\Join \{x,y\}$, hence $\{x,y\}$ is a face in $\Gamma(N_5)$.
The set $\{x,z\}\not\in\Gamma(N_5)$ because $\Join \{x,z\} = \1$ is \textit{not} the canonical join representation for $\1$, and $\{y,z\}$ is not a face in $\Gamma(N_5)$ because $\Join \{y,z\}$ is not even irredundant.
\end{example}

Our second application characterizes the collections of subsets $A$ of \textit{completely} join-irreducible elements that belong to $\Gamma(\tors \Lambda)$.
For the second statement of the theorem below, recall that $\tors \Lambda$ is finite if and only if $\Lambda$ is $\tau$-rigid finite (and in this case $\tors \Lambda$ is equal to the lattice of functorially finite torsion classes $\ftors \Lambda$).
See \cite{IRTT} and Remark~\ref{finite case}.
\begin{theorem}\label{main: cjr}
Suppose that $\calE$ is a collection of bricks over $\Lambda$.
Then the set $\{\filt(\Gen M):\, M\in \calE\}$ is a face of $\Gamma(\tors \Lambda)$ if and only if each pair $M$ and $M'$ satisfies the compatibility condition known as \newword{hom-orthogonality}:
\[\dim \Hom_{\Lambda}(M, M') = \dim \Hom_{\Lambda}(M', M) =0\]
In particular, if $\Lambda$ is $\tau$-rigid finite then $\Gamma(\tors \Lambda)$ is isomorphic to the complex of hom-orthogonal $\Lambda$ brick modules.
\end{theorem}

Finally, we consider the question:
Is the algebra $\Lambda$ characterized by its lattice of torsion classes?
For our third application, we give an extended counter-example.
We study the torsion theory of a certain quotient of the preprojective algebra in type $A_n$, which we call $RA_n$.
Unlike the preprojective algebra, $RA_n$ has finite representation type for each $n$.
We describe the canonical join complex of $RA_n$, and construct an explicit isomorphism from $\tors RA_n$ to the weak order on $A_n$. 
Thus, by Mizuno's result, $RA_n$ shares the same torsion theory as $\Pi A_n$.

%%%%% More %%%%%%
% More background here %
%%%%%%%%%%%%%%%%

\section{Minimal inclusions among torsion classes}
\label{sec:minim-incl-among}
In this section, we prove Theorems~\ref{main: covers} and~\ref{main: cor schur}.
In Proposition~\ref{thm:pseudo-minimal-covers}, we verify that the map $\eta_{\calT}$ from Theorem~\ref{main: covers} is well-defined.
That is, we argue that $\filt(\calT \cup \{M\})$ is indeed a torsion class, and that it covers $\calT$.
In Theorem~\ref{thm: eta inverse} we construct an inverse map of $\eta_\calT$.
The proof of the forward direction of Theorem~\ref{main: cor schur}
can be found in Lemma~\ref{lem: torsion and schur} while the remaining direction appears in the proof of Proposition~\ref{if part}.

\subsection{Preliminaries}\label{sec: background torsion classes}
Throughout, we take $\Lambda$ to be a finite-dimensional associative
algebra over a field $k$, and we write $\module\Lambda$ for the category of finite-dimensional (left) modules over $\Lambda$.
For $\calT$ a class of over $\Lambda$ (which we assume to be closed
under isomorphism), we write $\ind \calT$ for the set of indecomposable modules $M\in \calT$ and $[\ind\, \calT]$ for the set of isoclasses $[M]$ such that $M\in \ind \calT$.

A \newword{torsion class} $\calT$ is a class of modules that is closed under factors, isomorphisms, and extensions.
Dually, a class of modules $\calF$ is a \newword{torsion-free class} if it is closed under submodules, isomorphisms, and extensions.
As in the introduction, $\tors \Lambda$ denotes the lattice of torsion classes over $\Lambda$, in which $\calT\le \calT'$ if and only if~$\calT\subseteq \calT'$.
We write $\torf \Lambda$ for the lattice of torsion-free classes also ordered by containment.

At times it will be useful to translate a result or a proof from the language of torsion classes to the language of torsion-free classes.
To do this, we make use of the following standard dualities.
Given a torsion class $\calT$, denote by $\calT^\perp$ the set of modules $X\in
\module \Lambda$ such that $\Hom_\Lambda(\calT, X) = 0$.
The map $(-)^\perp: \tors \Lambda \rightarrow \torf \Lambda$ is a poset
anti-isomorphim with inverse given by $$\calF \mapsto ^\perp
\calF:=\{X\in \module \Lambda:\, \Hom_\Lambda(X, \calF)=0\}.$$
The duality functor $D=\Hom_k(-,k): \module \Lambda \rightarrow \module \Lambda^{\op}$ also furnishes a poset anti-isomorphism from $\tors
\Lambda$ to $\torf \Lambda^{\op}$ with inverse given by the same duality.
(Recall that $\Lambda^{\op}$ denotes the opposite algebra of
$\Lambda$. For details, see \cite{IRTT} or \cite{S}.)

Let $\calS$ be a set of indecomposable modules in $\module\Lambda$.
An \newword{$\calS$-filtration (of length $l$)} of a module $M$ in $\module\Lambda$ is a sequence of
submodules \[M=M_l \supsetneq M_{l-1} \supsetneq \dotsc \supsetneq M_1
\supsetneq M_0=0\] such that $M_i/M_{i-1}$ is isomorphic to a module in $\calS$ for each $i=1,\dotsc, l$.
We write $\filt^{(l)}(\calS)$ for the class of modules $M$ that admit an $\calS$-filtration of length at most~$l$ and $\filt(\calS)$ for the class of all modules in $\module \Lambda$ that admit an $\calS$-filtration, i.e.,
$\filt(\calS)=\bigcup_{l\geq 0} \filt^{(l)}(\calS)$.
When $\calS$ is not a class of indecomposable modules (but rather an additive full subcategory, for instance) we abuse the notation and write $\filt(\calS)$ instead of $\filt(\ind S)$.
We close this subsection with two lemmas.
%Both statements are valid for abelian categories, but we will state them here in the context of modules.
%The notion $\filt(\calS)$ can be generalized for an additive full subcategory $\mathcal C$ in $\module\Lambda$ which is closed under summands by defining $\filt(\calC):=\filt(\ind\calC)$, where $\ind\calC$ denotes a set of representatives isoclasses of all indecomposable modules in $\calC$.

\begin{lemma}\label{lem-ext-closure}
  Let $\calS$ be a class of indecomposable modules in $\module\Lambda$. The class $\filt(\calS)$ is closed under extensions.
\end{lemma}

\begin{proof} Suppose there is an exact sequence
 \[
 0\rightarrow N\rightarrow M\stackrel{\pi}\rightarrow N^\prime\rightarrow 0
 \]
 with $N\in \filt^{(l)}(\calS)$ and $N^\prime\in\filt^{(l^\prime)}(\calS)$.

 Let
 $N=N_l \supsetneq N_{l-1} \supsetneq \dotsc
 \supsetneq N_0 = 0$ and $N^\prime=N^\prime_{l^\prime} \supsetneq N^\prime_{l^\prime-1} \supsetneq \dotsc
 \supsetneq N^\prime_0 = 0$ be an $\calS$-filtration of $N$ and
 $N^\prime$ respectively. Take $M_{i}$ to be the pull back following, for $0\leq
 i\leq l^\prime$:

\begin{tikzpicture}[scale=0.75]
  \node (A) at (0,0) {$M_i$}; 
  \node(B) at (2,0) {$N_i'$};
  \node (C) at (0,-2) {$M$}; 
  \node (D) at (2,-2) {$N'$};
  \draw[->] (A.east)--(B.west);
  \draw[->] (A.south)--(C.north);
  \draw[->] (C.east)--(D.west);
  \draw[right hook->] (B.south)--(D.north);
\end{tikzpicture}\\
Then $M=M_{l^\prime} \supsetneq M_{l^\prime-1}\supsetneq \dotsc M_0=N=N_l \supsetneq N_{l-1} \supsetneq \dotsc
 \supsetneq N_0 = 0$ is an $\calS$-filtration of $M$.
\end{proof}

\begin{lemma}\label{lem-epi-closure}
Suppose that $\calS$ is a class of indecomposable modules that is
closed under taking indecomposable summands of factors.
Then $\filt(\calS)$ is closed under factors.
Dually, if $\calS$ is closed under taking indecomposable summands of submodules, then $\filt(\calS)$ is closed under submodules.
\end{lemma}
\begin{proof}
  Suppose that $N\in \filt^{(l)}(\calS)$ and $\phi: N\rightarrow U$ is an epimorphism.
We assume, without loss of generality, that $U$ is indecomposable and proceed by induction on $l$.
If $l=1$, then $N\in \calS$.
% so $U$ is an indecomposable factor of $N$.
Since $\calS$ is closed under epimorphisms, we have $U\in \calS$.
Now suppose that all factors of modules in $\filt^{(l^\prime)}(\calS)$
are in $\filt(\calS)$ for $l'<l$.
Since $N\in \filt^{(l)}(\calS)$, it admits an $\calS$-filtration $N=N_l\supsetneq N_{l-1} \supsetneq \dotsc
\supsetneq N_0=0$ with $N_i/N_{i-1}$ isomorphic to an object in
$\calS$. Consider the submodule $\phi(N_{l-1})$ of $U$, and notice that $\phi(N_{l-1})$ is a factor of $N_{l-1}\in \filt^{(l-1)}(\calS)$.
Therefore, by induction, $\phi(N_{l-1})$ has an $\calS$-filtration.
% \[ \phi(N_{l-1})=U_t\supsetneq U_{t-1} \supsetneq \dotsc \supsetneq U_0 = 0.\]
Further, $U/\phi(N_{l-1})$ is a factor of $N_{l}/N_{l-1}\in\calS$.
Therefore $U/\phi(N_{l-1})$ belongs to $\calS$.
We have the following short exact sequence:
 \[
 0\rightarrow \phi(N_{l-1})\rightarrow U\stackrel{\pi}\rightarrow U/\phi(N_{l-1})\rightarrow 0.
 \]
Lemma~\ref{lem-ext-closure} then implies that $U$ is in $\filt(\calS)$, as desired.

%We prove the dual statement similarly by induction.
%Suppose that $U$ is a submodule of $N$.
%By induction, $U\cap N_{l-1}$ is a member of $\filt(\calS)$, and also $U/(U\cap N_{l-1})$ is a submodule of $N_l/N_{l-1}$.
%Since $\calS$ is closed under submodules and $N_{l}/N_{l-1}$ is a member of $\calS$, we have $U/(U\cap N_{l-1})$ is a member of $\calS$.
%We have the short exact sequence:
%\[ 0 \rightarrow U\cap N_{l-1} \rightarrow U \rightarrow U/(N_{l-1}\cap U) \rightarrow 0.\]
%As above, we conclude that $U\in \filt(\calS)$.

\end{proof}

\subsection{Minimal extending modules}\label{proof sec}
In this subsection, we prove Theorem~\ref{main: covers}.
We begin by showing that $\eta_\calT: [M]\mapsto \filt(\calT \cup \{M\})$ is well-defined.
The first statement in the next lemma confirms that $\filt(\calT \cup \{M\})$ is indeed a torsion class.
Because it follows immediately, we also dispense with the forward direction of Theorem~\ref{main: cor schur}.
\begin{lemma}\label{lem: torsion and schur}
Suppose that $\calT\in \tors \Lambda$ and $M$ is an indecomposable $\Lambda$ module that does not belong to $\calT$.
If every proper factor of $M$ lies in $\calT$ then:
\begin{enumerate}
\item $\filt(\calT \cup \{M\})$ is a torsion class and
\item $M$ is a brick over $\Lambda$.
\end{enumerate}
\end{lemma}
\begin{proof}
The first statement follows immediately from Lemma~\ref{lem-epi-closure}.
To prove the second statement, let $f:M\rightarrow M$ be a non-zero morphism, and assume that $f$ is not an isomorphism.
%We need to show that $f$ is an isomorphism.
Consider the exact sequence:
$$
0\rightarrow \image f\rightarrow M\rightarrow M/\image f\rightarrow 0.
$$
Observe that both $\image f$ and $M/\image f$ are proper factors of $M$, and thus belong to $\calT$.
%Otherwise $0\neq\image f\subsetneq M$ is a proper factor of $M$. By assumption, $\image f\in \calT$.
%Since $\image f\neq 0$, $M/\image f$ is a proper factor of $M$. Hence $M/\image f\in\calT$.
Because $\calT$ is closed under extensions, it follows that $M\in\calT$, and that is a contradiction.
Therefore  $\End(M)$ only contains automorphisms
and hence  $\End(M)$ is a division ring. So $M$ is a brick.
\end{proof}

\begin{example}\label{noncover}
\normalfont
In light of Lemma~\ref{lem: torsion and schur}, we might be tempted to think that condition~\ref{propertyone} is sufficient for constructing cover relations in $\tors \Lambda$.
For a nonexample, consider the torsion class $\calT = \{3, \substack{3\\2}\}$ in $\tors kQ$ for the quiver $Q= 3\longrightarrow 2 \longrightarrow 1$.
(We write $i$ for the simple at $i$, and $\substack{3\\2}$ for the representation $k\overset{1}{\longrightarrow}k \longrightarrow 0$.)
The simple module at vertex $1$ trivially satisfies the condition that all of its proper factors belong to $\calT$.
We have following chain of torsion classes in $\tors \Lambda$:
$$\calT = \add\{3, \substack{3\\2}\} \subsetneq \add\{3, \substack{3\\2}, \substack{3\\2\\1} \}\subsetneq \add\{3, \substack{3\\2}, \substack{3\\2\\1}, 1 \} = \filt(\calT\cup \{1\}).$$
\end{example}

Below (Proposition~\ref{prop: epis and covers}) we establish precisely when $\filt(\calT\cup \{M\})$ covers $\calT$.

\begin{lemma}\label{uniqueness}
Let $\calT\in \tors\Lambda$ and $M\not\in \calT$ be an indecomposable module such that each proper factor of $M$ belongs to $\calT$.
Let $N\in \filt(\calT \cup \{M\})\setminus \calT$ such that each proper factor  $N$ lies in $\calT$.
If $\filt(\calT \cup \{M\})\covers \calT$ then $N\cong M$.
%Suppose that $N\in \filt(\calT \cup \{M\})\setminus \calT$ satisfies each proper factor  $N$ lies in $\calT$.
%Then $N$ is isomorphic to $M$.
 
\end{lemma}
\begin{proof}
Write $\calT'$ for $\filt(\calT \cup \{M\})$.
Note that $N$ also satisfies the assumptions of Lemma \ref{lem:
  torsion and schur}, so $\filt(\calT \cup \{N\})$ is a torsion class
which properly contains $\calT$ (since $N\notin \calT$) and is
contained in $\calT'$.
Hence, $\calT'=\filt(\calT \cup \{N\})$. 
In particular, $N$ admits a $\calT \cup \{M\}$-filtration with at
least one subfactor isomorphic to $M$, so $\dim N \geq \dim M$. 
Symmetrically, $M$ admits a $\calT \cup \{N\}$-filtration with at
least one subfactor isomorphic to $N$.
Therefore, $M\cong N$. 
% Consider the $\ind \calT \cup \{M\}$-filtration $N=N_l\supsetneq N_{l-1} \supsetneq \dotsc
% \supsetneq N_0=0.$
% Because $N\not\in \calT$, there is some index $i\in \{1,\ldots, l\}$
% such that $N_i/N_{i-1} \cong M$.
% In particular, since $\dim N = \sum_i \dim (N_i/N_{i-1})$, $\dim N\geq
% \dim M$. By Lemma \ref{lem: torsion and schur}, $\filt(\calT \cup
% \{N\})$ is also a torsion class, properly containing $\calT$, and
% contained in $\calT'$. Hence, $\filt(\calT\cup\{N\})=\calT'$.
% In particular, $M$ has a $\calT\cup \{N\}$-filtration $M=M_

% The (nonzero) factor $N/N_{i-1}$ has the filtration $N/N_{i-1}\supsetneq \dotsc
% N_{i}/N_{i-1}\supsetneq N_{i-1}/N_{i-1}$, and thus $N/N_{i-1}$ is also
% not in $\calT$ (because it has a composition factor isomorphic
% to~$M$).
% We conclude that $N/N_{i-1}\cong N$ and $M \cong N_1$.
% Observe that $\calT <\filt(\calT \cup \{N\}) \le \calT'$.  Since $\calT\covered \calT'$, we conclude that $\calT'=\filt(T\cup \{N\})$.
% Thus, we can replace ``$N$'' with ``$M$'' above, and obtain $N$ is a submodule of $M$.
\end{proof}

\begin{proposition}\label{prop: epis and covers}
Let $\calT\in \tors\Lambda$ and $M\not\in \calT$ be an indecomposable module such that each proper factor of $M$ belongs to $\calT$.
Then $\calT' =\filt(\calT \cup \{M\}) \covers \calT$ if and only if $M$ is a factor of each $N\in \calT'\setminus \calT$.
\end{proposition}
\begin{proof}
First we argue that``only if'' direction.
Suppose that every element in $\calT'\setminus \calT$ admits an epimorphism to $M$.
Let $\calG$ be a torsion class with $\calT \subsetneq \calG \subseteq \calT'$, and pick $N\in \calG\setminus \calT$.
Because $M$ is a factor of $N$, we have $M\in \calG$.
Thus, $\calT'=\filt(\calT \cup \{M\}) \subseteq \calG$, as desired.

Conversely, suppose that $\calT'\covers \calT$, and let $N$ be any module in $\calT'\setminus \calT$.
Among all submodules $N'\subseteq N$ such that $N/N'\in \calT'\setminus \calT$, choose $N'$ maximal.
By Lemma~\ref{uniqueness} $N/N'$ is isomorphic to $M$.
We conclude $M$ is a factor of $N$.
\end{proof}

We use the previous proposition to verify that $\eta_\calT$ (from Theorem~\ref{main: covers}) has the correct codomain.
\begin{proposition}\label{thm:pseudo-minimal-covers}
Suppose that $\calT\in \tors \Lambda$ and $M$ is an indecomposable $\Lambda$ module.
If $M$ is a minimal extending module for $\calT$, then $\filt(\calT \cup \{M\})\covers \calT$.
That is, the map $\eta_\calT: \ME(\calT) \to \covup(\calT)$ is well-defined.
\end{proposition}
\begin{proof}
Again, denote by $\calT'$ the torsion class $\filt(\calT\cup \{M\})$. Assume, without loss of generality, that $N$ is indecomposable. Properties~\ref{propertyone} and~\ref{propertythree} imply that $M$ satisfies the hypotheses of Proposition~\ref{prop: epis and covers}.
Let $N\in \calT' \setminus \calT$ with $\ind \calT \cup \{M\}$-filtration $N=N_l\supsetneq N_{l-1} \supsetneq \dotsc \supsetneq N_0=0$.
We argue by induction on $l$ that $M$ is a factor of $N$.
If $l=1$, then $N\cong M$ since $N\notin \calT$.
Suppose that $l>1$, and let $i$ to be the smallest index such that $N_i/N_{i-1} \cong M$ (one must exist since $N\notin
  \calT$).

If $i=1$, then there is a short exact sequence 
\[0 \rightarrow M \rightarrow N \rightarrow N/N_1 \rightarrow 0.\]
Since $N$ was assumed indecomposable (and $l>1$), the sequence does
not split.
If $N/N_1 \in \calT$, then by Property~\ref{propertytwo}, $N\in
\calT$, a contradiction.
Otherwise, $N/N_1 \in \calT$ is in $\filt^{(l-1)}(\calT\cup \{M\})$,
so by induction, $M$ is a factor of $N/N_1$, so it is also a factor of
$N$. 

If $i>1$, consider the short exact sequence 
\[0 \rightarrow N_{i-1} \rightarrow N \rightarrow N/N_{i-1} \rightarrow 0\]
which again is non-split by assumption.
Note that since $N_{i-1}$ has a filtration by modules in $\calT$, $N_{i-1}\in
\calT$. 
If $N/N_{i-1}\in \calT$, then so is $N$, since torsion classes are
closed under extensions. 
Therefore, $N/N_{i-1} \notin \calT$, and it has a
$\calT\cup\{M\}$-filtration of length $l-1+1<l$.
Hence, by induction, $M$ is a factor of $N/N_{i-1}$, and therefore of $N$ as well. 

Hence, $M$ is a factor of $N$ for each module $N$ in
$\calT'$, so $\calT\lessdot \calT'$. 
\end{proof}

Below we construct the inverse map to $\eta_\calT$, completing the proof of Theorem~\ref{main: covers}.
Recall that $[\ind\, \Lambda]$ is the set of isoclasses $[M]$ such that $M\in \ind \Lambda$.
\begin{theorem}\label{thm: eta inverse}
For each $\calT'\covers \calT$ in $\tors\Lambda$, there exists a unique (up to isomorphism) indecomposable module $M$ such that $\calT' = \filt(\calT \cup \{M\})$.
Furthermore, the map $\filt(\calT \cup \{M\}) \mapsto [M]$ is the inverse to $\eta_\calT$.
\end{theorem}
\begin{proof}[Proof of Theorem~\ref{thm: eta inverse} and Theorem~\ref{main: covers}]
Let $N\in \calT' \setminus \calT$.
As in the proof of Proposition~\ref{prop: epis and covers}, among all submodules $N'$ of $N$ such that $N/N'\in \calT'\setminus \calT$, choose $N'$ maximal.
We take $M$ to be the factor~$N/N'$.
(It is immediate that $M$ is indecomposable.)
Because $\calT'\covers \calT$, we conclude that $\calT' = \filt(\calT \cup \{M\})$.
The same argument as given in the third paragraph of the proof of Proposition~\ref{prop: epis and covers} shows that $M$ is unique up to isomorphism.

To prove the second statement, it is enough to show that $M$ is a minimal extending module.
It is immediate that $M$ satisfies Property~\ref{propertyone}.
Suppose that $0\to M \stackrel{i}\rightarrow X\stackrel{\pi}\rightarrow T\rightarrow 0$ is a short exact sequence with $T\in \calT$.
If $X\in \calT'\setminus \calT$, then Proposition~\ref{prop: epis and
  covers} implies that  there is an epimorphism $q: X\rightarrow
M$. If $q\circ i=0$, then $q$ factors through $\pi$ which implies that
$M$ is a quotient of $T$, contradicting with $M\notin \calT$. If
$q\circ i\neq 0$, then $q\circ i$ is an isomorphism due to the fact
that $M$ is a brick. So the exact sequence splits.
Thus~$M$ satisfies property~\ref{propertytwo}.
Assume that $Q\in \calT$ is a module with a non-zero morphism $f:Q\rightarrow M$.
Without loss of generality we can assume $f$ is a monomorphism, otherwise taking $Q=\image f$.
Then we have the canonical exact sequence\[0\rightarrow Q \rightarrow M \rightarrow M/Q\rightarrow 0\] in which $M/Q$ is a proper factor of $M$, and therefore lies in $\calT$.
Thus $M\in \calT$ (because $\calT$ is closed under extensions), and that is a contradiction.
We conclude that $M$ satisfies~\ref{propertythree}.
\end{proof}
The following corollary will be useful as we explore the connection to
the canonical join complex of $\tors \Lambda$ in
Section~\ref{canonical join complex}.
\begin{corollary}\label{hom orthogonality}
Let $\calT_1$ and $\calT_2$ be distinct torsion classes in $\tors \Lambda$, and for each $i\in \{1,2\}$, let $M_i$ be a minimal extending module for $\calT_i$.
If $\filt(\calT_1 \cup \{M_1\}) = \filt(\calT_2 \cup \{M_2\})$ then \[\dim \Hom_{\Lambda}(M_1,M_2) = \dim \Hom_{\Lambda}(M_2,M_1)=0.\]
\end{corollary}
\begin{proof}
Write $\calT'$ for $\filt(\calT_1 \cup \{M_1\}) = \filt(\calT_2 \cup \{M_2\})$.
First, we claim that $M_1$ belongs to $\calT_2$ and $M_2$ belongs to $\calT_1$.
Then the statement follow immediately from Property~\ref{propertythree} of minimal extending modules.
Since $\calT_1$ and $\calT_2$ are distinct and both covered by the same torsion class $\calT'$, it follows that there exists some $N\in \calT_1\setminus \calT_2$.
In particular, $N\in \calT'\setminus \calT_2$.
Proposition~\ref{prop: epis and covers} implies that $N$ surjects onto $M_2$.
Thus, $M_2\in \calT_1$.
By symmetry $M_1\in \calT_2$.
\end{proof}

\subsection{Torsion-free classes and bricks}
We now record the corresponding notions for
torsion-free classes, both for completeness, and for a convenient
proposition relating the upper covers of $\calT$ to the lower
covers of $\calT^\perp$.
(See Proposition~\ref{prop:torsion-free-to-torsion-relation} below.)
Proofs that are essentially equivalent to their counterparts in the
torsion context are suppressed.
We begin by defining the analogue to minimally extending modules.

\begin{definition}
A module $M$ is called a \newword{minimal co-extending module for a torsion-free class $\calF$} if it satisfies the following
three conditions:
\begin{enumerate}[label={(P\arabic*')}]
\item \label{copropertyone} Every proper submodule of $M$ is
  in $\calF$;
\item \label{copropertytwo} if $0\rightarrow F \rightarrow X
  \rightarrow M \rightarrow 0$ is a non-split exact sequence, with
  $F\in \calF$, then $X\in \calF$.
\item\label{copropertythree} $\Hom_\Lambda(M,\calF)=0$
\end{enumerate}
\end{definition}

\begin{theorem}\label{prop:min-ext-module-free}
Suppose that $\calF$ is a torsion-free class over $\Lambda$ and $M\not\in \calF$ is an indecomposable $\Lambda$ module.
The following are equivalent:
\begin{enumerate}
\item $M$ is a minimal co-extending module of $\calF$.
\item Each proper submodule of $M$ lies in $\calF$ and $M$ is a submodule for each $N\in \filt(\calF\cup \{M\})\setminus \calF$.
\item $\filt(\calF \cup\{M\}) \covers \calF$.
\end{enumerate}
The map $\zeta: [M]\mapsto \filt(\calF\cup \{M\})$ is a bijection from the set of isoclasses $[M]$ such that $M$ is a minimal co-extending module for $\calF$ to the set $\covup(\calF)$.
\end{theorem}

\begin{proposition}\label{prop:torsion-free-to-torsion-relation}
Suppose that $\calT$ is a torsion class in $\tors \Lambda$ and $M$ is
an indecomposable $\Lambda$ module such that $\filt(\calT \cup \{M\})$
is a torsion class. Then $M$ is a minimal extending module for the torsion class $\calT$ if and only if it is a minimal co-extending module for the torsion-free class $\filt(\calT \cup\{M\})^\perp$.
\end{proposition}

\begin{proof}
We argue the forward implication of the proposition; the reverse implication is similar. 
Let $\calT'$ denote the torsion class $\eta_\calT([M])=\filt(\calT \cup \{M\})$.
Observe that $M\not\in (\calT')^\perp$.
Also, Property~\ref{propertythree} implies that $M$ belongs to $\calT^\perp$.
By Theorem~\ref{prop:min-ext-module-free}, it is enough to show that $M$ satisfies the following:
First, each proper submodule of $M$ belongs to $\calT^\perp$; and second, $M$ is a submodule of each $X\in (\calT')^\perp \setminus \calT^\perp$.

Suppose that $M'$ is an indecomposable submodule of $M$, and $N$ belongs to $\calT'\setminus \calT$.
We claim that $\Hom_\Lambda(N, M')$ is nonzero only if $M'=M$.
Let $\phi: N\rightarrow M'$ be such a nonzero homomorphism.
%We may assume without loss of generality that $\phi$ is surjective (if it is not then replace $M'$ with the submodule $\phi(N)$).
On one hand, $\image \phi \in \calT'$ (because $\calT'$ is closed under epimorphisms).
On the other hand, $\image\phi \in \calT^\perp$ (because torsion-free classes are closed under submodules).
In particular, $\image\phi\not\in \calT$.
By Proposition~\ref{prop: epis and covers}, there is a surjection of $\image \phi$ onto $M$.
Because $\image\phi$ is a submodule of $M$, we obtain $\image\phi\cong M'\cong M$, as desired.
We conclude that each proper submodule of $M$ belongs to
$(\calT')^\perp$.

Suppose that $X\in \calT^\perp \setminus (\calT')^\perp$, and let $f: N\to X$ be a nonzero morphism from a module $N\in \calT'\setminus \calT$.
We may assume that $f$ is injective.
(If it is not injective, then replace $N$ with $N/\ker(f)$.)
We claim that $M$ is a submodule of $N$. 
Let $N=N_l \supsetneq \dotsc \supsetneq N_0=0$ be an $\ind(\calT)\cup \{M\}$-filtration of $N$.
Since $N\notin \calT$, there is some index $i$ such that $N_i/N_{i-1}$ is isomorphic to $M$.
For the moment, assume that $i>1$, so that $N_1\in \calT$.
Then we have a nonzero homomorphism $N_1\hookrightarrow N \hookrightarrow X$ from a module in $\calT$ to $X$.
That is a contradiction.
Thus, $M\cong N_1 \hookrightarrow N \hookrightarrow X$ as desired.
\end{proof}

We close this section by completing the proof of Theorem~\ref{main: cor schur}.
Recall that in Lemma~\ref{lem: torsion and schur}, we showed that if
$M$ is a minimal extending module, then it is a brick (the forward implication in Theorem~\ref{main: cor schur}).
By Proposition~\ref{prop:torsion-free-to-torsion-relation} it is enough to show:
If $M$ is a brick, then there exists a torsion-free class $\calF$ such that $M$ is a minimal co-extending module for $\calF$.
\begin{proposition}\label{if part}
If $M$ is a brick over $\Lambda$, then $M$ is a minimal co-extending module for $\filt(\Gen M)^{\perp}$.
\end{proposition}
\begin{proof}[Proof of Proposition~\ref{if part} and Theorem~\ref{main: cor schur}]

Suppose that $M$ is a brick over $\Lambda$.  We will argue that $M$ is a minimal co-extending module for the torsion-free class $\calF=\filt(\Gen M)^\perp$.

First, we claim that $\filt(\Gen M)^\perp$ coincides with the set $\{X:\,\Hom_\Lambda(M, X)=0\}$.
(This claim is equivalent to~\ref{copropertythree}.)
It is immediate that $\filt(\Gen M)^\perp \subseteq \{X:\,\Hom_\Lambda(M, X)=0\}$.
Suppose that $X$ is an indecomposable module satisfying:
$\Hom_\Lambda(M,X)=0$ and $\Hom_\Lambda(N, X)\neq 0$ for some
indecomposable $N\in \filt(\Gen M)$.
Write $f: N\rightarrow X$ for such a non-zero homomorphism.
Let $N_l \supsetneq \dotsc \supsetneq N_0=0$ be a filtration of $N$ such that $N_i/N_{i-1}$ is a factor of $M$.
Choose $i$ to be the smallest index for which $f(N_i)\neq 0$, so that $f(N_i)/f(N_{i-1})=f(N_i)\subseteq X$.
Since is $N_i/N_{i-1}$ is a factor of $M$, we have the following nonzero composition:
\[M \twoheadrightarrow N_i/N_{i-1} \twoheadrightarrow f(N_i) \subseteq X.\]
That is a contradiction, because $\Hom_\Lambda(M,X) = 0$.

To verify ~\ref{copropertyone}, let $M'$ be a proper submodule of
$M$. If $\Hom_\Lambda(M,M')\neq 0$, then $M$ is not a brick, a contradiction. Thus, $M'\in\{X:\,\Hom(M,X)=0\}=\filt(\Gen M)^\perp$.  

To verify~\ref{copropertytwo}, suppose that there is a non-split short exact sequence \[ 0\rightarrow F \xrightarrow{f} X
\xrightarrow{g} M \rightarrow 0\] with $F\in \calF$.
If $X\notin \calF$, then there is a nonzero homomorphism $\pi: M\rightarrow X$.
Since $M$ is a brick, and $g\circ \pi$ is a endomorphism of $M$ (which is not an isomorphism since the sequence is non-split), $g\circ \pi=0$. 
Hence, $\pi$ factors through $f$.
Since $\Hom_\Lambda(M,F)=0$, we have a contradiction.

Therefore, $M$ is a minimal co-extending module for $\calF$.
The statement follows from Proposition~\ref{prop:torsion-free-to-torsion-relation}.
\end{proof}

\section{Applications}
We consider two applications of Theorems~\ref{main: covers} and~\ref{main: cor schur}.
First, in Propostion~\ref{ji}, we prove that there is a bijective
correspondence between isoclasses $[M]$ of bricks over $\Lambda$ and
completely join-irreducible torsion classes in $\tors \Lambda$ via the
map sending $[M]$ to $\filt(\Gen(M)).$
Second, we consider the canonical join complex of $\tors \Lambda$,
proving the forward implication of Theorem~\ref{main: cjr} in Proposition~\ref{cjr only if}, and we completing the proof in Proposition~\ref{cjr end}.

\label{sec: ji}
\subsection{Completely join-irreducible torsion classes}
In this section, we prove Theorem~\ref{schur modules to join-irreducibles} as Proposition~\ref{ji} below.
Recall that $\calT$ is \newword{completely join-irreducible} if for each (possibly infinite) subset $A\subseteq \tors \Lambda$, $\calT = \Join A $ implies that $\calT\in A$.
Equivalently, $\calT$ is completely join-irreducible if and only if it covers precisely one torsion class $\calS$, and $\calG\subseteq \calS$ for any torsion class $\calG\subsetneq \calT$.
In the statement below, recall that $\Gen M$ denotes the factor-closure of $\add M$.
\begin{proposition}\label{ji}
The torsion class $\calT$ is completely join-irreducible if and only
if there exists a brick $M$ such that $\calT$ is equal to $\filt(\Gen M)$.
In this case, the brick $M$ is unique up to isomorphism. 
In particular, the map $\zeta: [M]\mapsto \filt(\Gen M)$ from
Theorem~\ref{schur modules to join-irreducibles} is a bijection. 
\end{proposition}

\begin{proof}[Proof Proposition~\ref{ji} and Theorem~\ref{schur modules to join-irreducibles}]
Suppose that the torsion class $\calT$ is completely join-irreducible
and write $\calS$ for the unique torsion class covered by $\calT$.
Theorem~\ref{main: cor schur} implies that there exists a unique (up to isomorphism) brick $M$ such that ${\calT = \filt(\calS \cup \{M\})}$.
We claim that $\calT = \filt(\Gen M)$.
Observe that $\filt(\Gen M)$ is contained in $\calT$, and $\filt(\Gen M) \not\subseteq\calS$ (because $M\not\in \calS$).
The claim follows.

Conversely, let $M$ be a brick over $\Lambda$.
In the proof of Proposition~\ref{if part}, we showed that $M$ is a minimal co-extending for $\filt(\Gen M)^\perp$.
By Proposition~\ref{prop:torsion-free-to-torsion-relation}, there exists a torsion class $\calS$ such that $\filt(\Gen M) \covers \calS$ and $M$ is a minimal extending module for $\calS$.
(That is $\filt(\Gen M) = \filt(\calS \cup \{M\})$.)
Suppose that $\calG\subseteq \filt(\Gen M)$.
If $\calG\not \subseteq \calS$ then there exists some module $N\in \calG\setminus \calS$.
In particular, $N\in\filt(\Gen M) \setminus \calS$.
Proposition~\ref{prop: epis and covers} says that $M$ is a factor of $N$, hence $M\in \calG$.
We conclude that $\calG= \filt(\Gen M)$.
Thus any torsion class $\calG\subsetneq\filt(\Gen M)$ also satisfies
$\calG\subseteq\calS$. 

We conclude that $\filt(\Gen M)$ is completely join-irreducible.
\end{proof}

\begin{remark}\label{kronecker}
\normalfont
There exist join-irreducible torsion classes which are not \textit{completely} join-irreducible.
Consider $\module kQ$ where $Q$ is the Kronecker quiver and $k$ is
algebraically closed (see Figure~\ref{fig: kronecker}).
\begin{figure}[h]
\begin{tikzpicture}
\node (A) at (0,0) {1};
\node (B) at (2,0) {2};
\draw[->] ([yshift=-.2em] A.east)--([yshift=-.2em]B.west) node[midway,below] {b};
  \draw[->] ([yshift=.2em] A.east)--([yshift=.2em] B.west) node[midway,above]
  {a};
  \end{tikzpicture}
  \caption{The Kronecker quiver.}
  \label{fig: kronecker}
  \end{figure}
Let $n$ be a non-negative integer, and write $V_n$ for the representation defined as follows:
$V_n(1) = k^{n+1}$, $V_n(2) = k^n$; $V_n(a)=\begin{bmatrix} I_n \,\, \mathbf{0}\end{bmatrix}$ and $V_n(b)=\begin{bmatrix} \mathbf{0}\,\, I_n\end{bmatrix}$ where $I_n$ is the $n\times n$ identity matrix, and $\mathbf{0}$ is a column of zeros.
(The module corresponding to each $V_n$ is indecomposable and preinjective; see e.g. \cite[VIII.1]{ARS}.)
Let $\calI_n$ denote the additive closure of $\{V_0,V_1,\dotsc, V_n\}$, and $\calI_\infty=\bigcup\limits_{n\geq 0} \calI_n$.
It is an easy exercise to verify that both $\calI_n$ and $\calI_\infty$ are torsion classes and  $\calI_n \covered \calI_{n+1}<\calI_\infty$ for each $n$.
Observe that $\calI_\infty$ is join-irreducible, but not \textit{completely} join-irreducible.
In particular, it does not cover any elements in $\tors kQ$.
Each brick in $\calI_\infty$ is isomorphic to $V_n$ for some $n\geq
0$, and it can be shown that $\calI_\infty$ cannot be expressed as $\filt(\Gen V_n)$ for
any such $n$. 
\end{remark}

\subsection{The canonical join complex of $\tors \Lambda$}\label{canonical join complex}
In this section, we characterize certain faces of the canonical join complex of $\tors \Lambda$.
Before we begin, we review the necessary lattice-theoretic terminology.
Recall that a lattice $L$ is a poset such that, for each finite subset $A\subseteq L$, the \newword{join} or least upper bound $\Join A$ exists and, dually, the \newword{meet} or greatest lower bound $\Meet A$ exists.
The lattice $\tors \Lambda$ is a \newword{complete lattice}, meaning that $\Join A$ and $\Meet A$ exist for arbitrary subsets $A$ of torsion classes.
A subset $A\subseteq L$ is an \newword{antichain} if the elements in $A$ are not comparable.
The \newword{order ideal} generated by $A$ is the set of $w\in L$ such that $w\le a$ for some element $a\in A$.

Recall that the canonical join representation of an element $w$ is the unique ``lowest'' way to write $w$ as the join of smaller elements.
Below, we make these notions precise.
A \newword{join representation} of an element $w$ in a complete lattice is an expression $w=\Join A$, where $A$ is (possibly infinite) subset of $L$.
We say that $\Join A$ is \newword{irredundant} if $\Join A' < \Join A$, for each proper subset $A'\subsetneq A$.
Observe that if $\Join A$ is irredundant, then $A$ is an antichain.

Consider the set of all irredundant join representation for $w$.
(Note that $\Join \{w\}$ is an irredundant join representation of $w$.)
We partially order the irredundant join representations of $w$ as follows:
Say $A \refines B$ if the order ideal generated by $A$ is contained in the order ideal generated by $B$.
Equivalently, $A\refines B$ if, for each $a\in A$, there exists some $b\in B$ such that $a\le b$.
Informally, we say that $A$ is ``lower'' than~$B$.

The \newword{canonical join representation of $w$} is the unique lowest irredundant join representation $\Join A$ of $w$, when such a representation exists.
In this case, we also say that the set $A$ is a canonical join representation.
%We write $\Join \can(w)$ or simply $\can(w)$, for the canonical join representation of $w$.
The elements of $A$ are called \newword{canonical joinands} of $w$.
If $w$ is join-irreducible then $\Join \{w\}$ is the canonical join representation of $w$.
Conversely, each canonical joinand is join-irreducible.
\begin{figure}[h]
  \centering
 {\scalebox{.8}{ \includegraphics{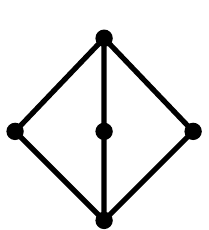}}}
  \caption{The top element does not have a canonical join representation}
  \label{fig:diamond}
\end{figure}

\begin{remark}\label{cjr dne}
\normalfont
In general, the canonical join representation of an element may not exist.
For example, see Figure \ref{fig:diamond}.
Observe that the join of each pair of atoms is a minimal, irredundant join representation of the top element.
Since there is no \textit{unique} such join representation, we conclude that the canonical join representation does not exist.
\end{remark}

The \newword{canonical join complex of $L$}, denoted $\Gamma(L)$, is the collection of subsets $A\subseteq L$ such $A$ is a canonical join representation.
In following proposition (essentially \cite[Proposition~2.2]{arcs}) we show that $\Gamma(L)$ is closed under taking subsets.
In the statement of \cite[Proposition~2.2]{arcs}, the lattice $L$ is finite.
A standard argument from lattice theory shows that the proposition also holds when $L$ is a complete meet-semilattice.
\begin{proposition}\label{complex}
Suppose that $\Join A$ is a canonical join representation in a complete meet-semilattice $L$.
Then, for each $A'\subseteq A$, the join $\Join A'$ is also canonical join representation.
\end{proposition}
%\begin{proof}
%Write $w$ for the element $\Join A$.
%Observe that $\Join A'$ exists for any $A'\subseteq A$.
%Indeed that the set $S= \{v \in L: \text{$v\ge a$ for each $a \in A'$}\}$ is nonempty.
%(In particular, $w$ belongs to this set.)
%Taking the meet $\Meet S$ yields the desired element in $L$.
%
%The remainder of the argument is identical to the original proof of \cite[Proposition~2.2]{arcs}.
%Write $w'$ for $\Join A'$.
%Suppose that $\Join B$ is another join representation of $w'$.
%Observe that $\Join (A\setminus A') \cup B = w$.
%Thus, for each element $a \in A$, there exists some element $b \in (A\setminus A')\cup B$ satisfying $a \le b$.
%If $a \in A'$, then the corresponding element $b \in B$ (because $A$ is an antichain).
%Hence, $\Join A'$ is the unique lowest join representation for $w'$.
%\end{proof}
With Proposition~\ref{complex}, we conclude that the canonical join complex $\Gamma(\tors \Lambda)$ is indeed a simplicial complex.
We are finally prepared to prove Theorem~\ref{main: cjr}.
In the next proposition, we tackle the easier direction of the proof.
\begin{proposition}\label{cjr only if}
Suppose that $\calE$ is a collection of bricks over $\Lambda$.
If the set $\{\filt(\Gen M):\, M\in \calE\}$ is a canonical join representation, then each pair of modules $M$ and $M'$ is hom-orthogonal.
\end{proposition}
\begin{proof}
By Proposition~\ref{complex}, it is enough to show that the statement holds when $\calE$ contains two elements, say $M$ and $M'$.
We write $\calT$ for $\filt(\Gen M) \join \filt(\Gen M')$.

Suppose that $f:M'\to M$ is a nonzero homomorphism, and write $N$ for the quotient $M'/\ker f$.
Since $\filt(\Gen M) \join \filt(\Gen M')$ is irredundant, it follows that $N$ is not isomorphic to $M$.
Hence, $M/N$ is a proper factor of $M$, and the torsion class $\filt(\Gen(M/N))$ is strictly contained in $\filt(\Gen M)$.
(Indeed, if $M\in \filt(\Gen(M/N))$, then there exists a submodule $Y\subseteq M$ that is also a proper factor of $M$.
That is a contradiction to the fact that $M$ is a brick.)

Finally, we observe that $\filt(\Gen M/N) \join \filt(\Gen M')$ is an irredundant join representation for $\calT$, and because $\filt(\Gen M/N) \subseteq \filt(\Gen M)$,
$$\{\filt(\Gen M/N), \filt(\Gen M')\} \refines \{\filt(\Gen M), \filt(\Gen M')\}.$$ 
That is a contradiction to the fact that $\filt(\Gen M)\join \filt(\Gen M')$ is the canonical join representation for $\calT$.
\end{proof}

Next, we consider the reverse implication of Theorem~\ref{main: cjr}.
As above, let $\calE$ be a collection of hom-orthogonal bricks.
We write $\calT$ for the torsion class $\Join \{\filt(\Gen M):\, M\in
\calE\}$, and  argue that $\Join \{\filt(\Gen M):\, M\in \calE\}$ is the
canonical join representation of $\calT$.
This will require the following lemma. 
For context, recall that a torsion class need not have any upper or lower cover relations.
(For example, the torsion class $\calI_\infty$ in the Kronecker quiver
from Example~\ref{kronecker} does not cover any other torsion class.)

\begin{lemma}\label{lem:join-has-lower-faces}
Suppose that $\calE$ is a collection of hom-orthogonal bricks over
$\Lambda$ and let $\calT = \Join \{\filt(\Gen M):\,M\in \calE\}$.
Then for each module $M\in \calE$:
\begin{enumerate}
\item  there is a torsion class $\calS\covered \calT$ such that $\calT = \filt(\calS \cup M)$;
 \item $M$ is a minimal extending module for $\calS$; and
 \item  $\calS$ contains $\calE \setminus \{M\}$.
 \end{enumerate}
\end{lemma}

\begin{proof}
The torsion class $\calT=\Join \{\filt(\Gen(M)):\, M\in \calE\}$ is the smallest torsion class in $\tors \Lambda$ that contains the modules in $\calE$.
Thus, a module $N$ belongs to $\calT$ if and only if $N$ admits a filtration $N=N_l \supsetneq  \dotsc \supsetneq N_0=0$ such that each $N_i/N_{i-1}$ is a factor of some indecomposable module in $\calE$.
For each $M\in \calE$, we claim that $M$ is a minimal co-extending module for the torsion-free class $\calT^\perp$.

First we check Property~\ref{copropertyone}.
Suppose that $Y$ is an indecomposable proper submodule of $M$ and that $Y\notin \calT^\perp$.
So, there exists a module $N\in \calT$ and a nonzero homorphism $f: N\rightarrow Y$.
We may assume, without loss of generality, that $f$ is injective (if not, we take the map $f:N/\ker f\rightarrow Y$, noting that $N/\ker f$ is in $\calT$ by closure under factors).
From the filtration of $N$ described above, observe that the submodule $N_1$ is a factor of $M'$, for some $M'\in \calE$.
Also $f(N_1)$ is a non-trivial submodule of $Y$.
So, we have the following sequence of homomorphisms \[ M' \onto N_1 \onto f(N_1) \subseteq Y\subseteq M.\]
The composition of these homomorphisms is non-zero, contradicting our hypothesis that $\dim \Hom_\Lambda (M', M) = 0$.
Therefore, $Y\in \calT^\perp$.
A similar argument shows that Property~\ref{copropertythree} also holds.

To verify Property~\ref{copropertytwo}, suppose that $0 \rightarrow F \xrightarrow{i} X \xrightarrow{\pi} M_1 \rightarrow 0$ is a non-split exact sequence with $F\in \calT^\perp$, and that $X\notin \calT^\perp$.
As above, let $f: N\rightarrow X$ be a non-zero homomorphism, where $N\in \calT$ indecomposable.
We may again assume that $f$ is injective, and in particular the restriction of the map to $N_1$ is injective.
As above, $N_1$ is a factor of some module $M'\in\calE$.
Thus, we have the nonzero composite homomorphism:
\[\tilde{f}: M' \onto N_1 \onto f(N_1) \subseteq X.\]
Since $\dim \Hom_{\Lambda}(M', M) =0$ the composition $\pi\circ \tilde{f}$ is zero.
Therefore, $\tilde{f}$ factors through the module $F$, and we have a nonzero homomorphism from $M'$ to $F$.
Since $F\in \calT^\perp$, that is a contradiction.
We conclude that $X\in \calT^\perp$, and $M$ a minimal co-extending module for~$\calT^\perp$.
Proposition~\ref{prop:torsion-free-to-torsion-relation} implies that there exists a torsion class $\calS\covered \calT$ such that $\calT= \filt(\calS \cup \{M\})$ and $M$ a minimal extending module for $\calS$.

To prove the third statement, suppose that $M'\in \calE\setminus \{M\}$ does not belong to~$\calS$.
Proposition~\ref{prop: epis and covers} implies that $M$ is a factor of $M'$, and that is a contradiction.
The statement follows.
\end{proof}

The following proposition completes our proof of Theorem~\ref{main: cjr}.
\begin{proposition}\label{cjr end}
Suppose that $\calE$ is a collection of hom-orthogonal $\Lambda$ brick modules and write $\calT$ for the torsion class $\Join \{\filt( \Gen M):\,M\in \calE\}$.
Then the expression $\Join \{\filt( \Gen M):\,M\in \calE\}$ is the canonical join representation of $\calT$.
\end{proposition}
\begin{proof}
We assume that $\calE$ has at least two elements (otherwise the statement follows from Theorem~\ref{schur modules to join-irreducibles}).
First we show that $\Join \{\filt(\Gen M):\,M\in \calE\}$ is irredundant.
Let $M\in \calE$, and consider the torsion class $\calT'= \Join \{\filt(\Gen M'):\,M'\in \calE\setminus \{M\}\}$.
By Lemma~\ref{lem:join-has-lower-faces}, there exists $\calS\covered \calT$ such that each module in $\calE\setminus \{M\}$ lies in $\calS$.
Thus $\calT' \le \calS \covered \calT$.

Next, we show the expression is $\Join \{\filt(\Gen M):\,M\in \calE\}$ is the unique lowest join representation for $\calT$.
Suppose that $\Join A$ is another irredundant join representation.
We claim that for each torsion class $\filt(\Gen M)$ such that $M\in \calE$, there exists some $\calG\in A$ such that $\filt(\Gen M) \le \calG$.
(Thus, $\{\filt(\Gen M):\,M\in \calE\}$ is ``lower'' than $A$.)
%It suffices to show that for each $M\in \calE$, there is a torsion class $\calG\in A$ such that $M\in \calG$.
Fix a module $M\in \calE$, let $\calS$ be the torsion class covered by $\calT$ satisfying $\calT = \filt(\calS \cup \{M\})$ and $M$ is a minimal extending module for $\calS$.
(Such a torsion class $\calS$ exists by Lemma~\ref{lem:join-has-lower-faces}.)
Observe that there exists $\calG\in A$ such that $\calG\not \subseteq \calS$.
Indeed, if each $\calG$ is contained in $\calS$, then $\Join A = \calS$.
Let $N\in \calG \setminus \calS$.
Proposition~\ref{prop: epis and covers} implies that $M$ is a factor of $N$, hence $M\in \calG$.
We conclude that $\filt(\Gen M) \le \calG$.
We have proved the claim and the proposition.
\end{proof}

\begin{remark}\label{finite case}
\normalfont
When $\tors \Lambda$ is finite (equivalently, when $\Lambda$ is $\tau$-rigid finite) each join-irreducible torsion class is completely join-irreducible.
Thus, the canonical join complex $\Gamma(\tors \Lambda)$ is isomorphic to the complex of hom-orthogonal bricks.
Moreover, the proof Proposition~\ref{cjr end} implies that the number of canonical joinands of $\calT\in \tors \Lambda$ is equal the number of torsion classes $\calS$ covered by $\calT$.
More precisely:

\begin{corollary}\label{cor: torsion labeling}
Suppose that $\calT$ is a torsion class over $\Lambda$ with the
following property: for every torsion class $\calS$ with
$\calS<\calT$, there is a torsion class $\calT'$ such that $\calS \leq
\calT' \lessdot \calT$. 
Then the canonical join representation of $\calT$ is equal to \[\Join \{\filt(\Gen(M)):\,\text{$M$ is a minimal co-extending module of $\calT^\perp$}\}.\]
In particular, if $\Lambda$ is $\tau$-rigid finite, each torsion class has a canonical join representation.
\end{corollary}
\begin{proof}
The statement follows immediately form Corollary~\ref{hom orthogonality}.

\end{proof}

Indeed, the canonical join representation ``sees'' the geometry of the Hasse diagram for \textit{any} finite lattice.
We summarize this useful fact below (see \cite[Proposition~2.2]{flag}).
\begin{proposition}\label{prop: labeling}
Suppose that $L$ is a finite lattice, and for each element $w\in L$ the canonical join representation of $w$ exists.
Then, for each $w\in L$, the number of canonical joinands of $w$ is equal to the number of elements covered by $w$.
\end{proposition}

\end{remark}
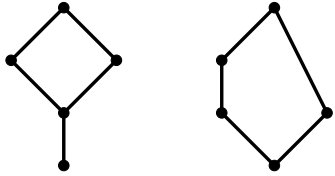
\begin{figure}[h]
\centering
\scalebox{.7}{
\begin{tikzpicture}
%distributive lattice
%\draw[help lines] (-3,0) grid (3,3);
\draw[fill] (-2,0) circle [radius=.1];
\draw[fill] (-2,1) circle [radius=.1];
\draw[fill] (-3,2) circle [radius=.1];
\draw[fill] (-1,2) circle [radius=.1];
\draw[fill] (-2,3) circle [radius=.1];

\draw [black, line width=.65 mm] (-2,0) to (-2,1);
\draw[black, line width=.65 mm] (-2,1) to (-3,2);
\draw [black, line width=.65 mm] (-2,1) to (-1,2);
\draw [black, line width=.65 mm] (-2,3) to (-3,2);
\draw[black, line width=.65 mm] (-2,3) to (-1,2);
%nondistributive
\draw[fill] (2,0) circle [radius=.1];
\draw[fill] (1,1) circle [radius=.1];
\draw[fill] (3,1) circle [radius=.1];
\draw[fill] (1,2) circle [radius=.1];
\draw[fill] (2,3) circle [radius=.1];

\draw [black, line width=.65 mm] (2,0) to (1,1);
\draw [black, line width=.65 mm] (2,0) to (3,1);
\draw [black, line width=.65 mm] (1,1) to (1,2);
\draw [black, line width=.65 mm] (2,3) to (1,2);
\draw [black, line width=.65 mm] (2,3) to (3,1);
\end{tikzpicture}}
\caption{Two nonisomorphic lattices with isomorphic canonical join complexes.}
\label{fig: noniso}
\end{figure}

\begin{remark}\label{isomorphism}
\normalfont
It is natural to ask if the canonical join complex $\Gamma(\tors \Lambda)$ characterizes the underlying algebra $\Lambda$ or the torsion theory of $\Lambda$.
It is easy to see that non-isomorphic algebras may have the same torsion theory.
We explore such an example in the following section when we show that the algebra $RA_n$ (an algebra of finite representation type for all $n$) has the same torsion theory as the preprojective algebra $\Pi A_n$ (which is \textit{not} finite representation type for $n\ge 4$).

Furthermore, nonisomorphic lattices $L$ and $L'$ may have isomorphic canonical join complexes.
For example, consider the two (nonisomorphic) lattices shown in Figure~\ref{fig: noniso}.
It is an easy exercise to check that canonical join complex of both lattices consists of an edge and an isolated vertex.
(See Example~\ref{Tamari example}.)
\end{remark}

\section{$\tors RA_n$ and the weak order on $A_n$}
\label{sec:example:-shards}

Mizuno showed (in \cite[Theorem~2.3]{Miz}) that the lattice of torsion classes for the preprojective algebra of Dynkin type $W$ is isomorphic to the weak order on the associated Weyl group, when $W$ is simply laced.
In the last section, we construct a different algebra, which we
refer to as $RA_n$, and show that $\tors RA_n$ is isomorphic to the weak
order on $A_n$. 
This is carried out in two steps:
First, in Theorem~\ref{thm: cjc isomorphism}, we show that the canonical join complex of $\tors RA_n$ is isomorphic to the canonical join complex of the weak order on $A_n$.
Second, in Proposition~\ref{inversion and cover}, we map each cover relation $\calT\covered \calT'$ in $\tors RA_n$ bijectively to a cover relation of permutations in weak order.
%We will show in Section~\ref{RA_n background} that each indecomposable module corresponds to an orientation of some (possibly smaller) type A Dynkin Diagram.
%The benefit of working with $RA_n$ in lieu of $\Pi A_n$ is that $RA_n$ is
%finite type for all $n$.
%Our main tool will be an isomorphism $\Gamma$ from the canonical join complex of $\tors(RA_n)$ to the canonical join complex of the weak order on~$A_n$.
Before we begin, we establish some useful notation.
Throughout we write $[n]$ for the set $\{1,2,\ldots, n\}$, $[i,k]$ for
the set $\{i, i+1, \ldots, k\}$ and $(i,k)$ for $\{i+1, \ldots,
k-1\}$.

\subsection{The algebra $RA_n$}\label{RA_n background}
Let $Q$ be the quiver with vertex set $Q_0=[n]$ and arrows $Q_1=\{a_i:
i\rightarrow i+1, a_i^*: i+1\rightarrow i\}_{i\in [n-1]}$. Define $I$ to be the two-sided ideal in the
  path algebra $\kk
  Q$ generated by all two-cycles, $I=\langle a_ia_i^*, a_i^*a_i \mid
  i\in [n-1]\rangle$, and define $RA_n$ to be the algebra $\kk Q/I$.

Recall that a representation of $(Q,I)$ is a collection of vector spaces $M(x)$, one for each
vertex $i\in Q_0$, and linear maps $M(a_i): M(i)\rightarrow M(i+1)$,
$M(a_i^*):M(i+1)\rightarrow M(i)$, for each of the arrows in $Q$, that satisfy the
relations given by $I$. We will make generous use of the equivalence between the category of
modules over $RA_n$ and that of the representations of the bound
quiver $(Q, I)$, generally referring to the objects of interest as
modules, while describing them as representations.
 The \newword{support} of a representation is the set of vertices $i$ for which $M(i)\neq 0$.

\begin{proposition}
There are finitely many isoclasses of  indecomposable representations
of $RA_n$ for each $n$.
\end{proposition}

\begin{proof}
The algebra $RA_n$ is gentle (see \cite{AS}) with no band modules since any cycle contains a 2-cycle,
each of which lies in $I$. By the work of Butler-Ringel \cite{Butler-Ringel},
then, there are finitely many isoclassses of indecomposable modules.
\end{proof}

As a quiver representation, each indecomposable module $M$ (up to
isomorphism) over $RA_n$ corresponds to a connected subquiver $Q_M$ of
$Q$ satisfying the condition that at most one of either $a_i$ or $a_i^*$ belongs to $(Q_M)_1$.
Thus, each indecomposable module can be identified with an orientation of a type-A Dynkin diagram with rank less than or equal to $n$.
More precisely, the quiver representation corresponding to $Q_M$ satisfies:
\begin{itemize}
\item $M(i)=k$ for all $i\in (Q_M)_0$ and $M(i)=0$ for $i\notin (Q_M)_0$;
\item $M(a)\neq 0$ if and only if $a\in (Q_M)_1$.
\end{itemize}

%Because these algebras are representation directed, we obtain the following proposition.

%What's more, we can parameterize these indecomposables fairly
%easily. In particular, there is a bijection between indecomposable
%$RA_n$ modules and connected subquivers $Q'\subseteq Q$ such that for
%each $i\in [n]$, at most one of $a_i, a_i^*$ in $Q'_1$. The module $M$
%corresponding to $Q'$ has $\dim M(x) = \begin{cases} 1 &
%  x\in Q'_0 \\ 0 & \textrm{otherwise}\end{cases}$, and $M(a)$ acts
%non-trivially if and only if $a\in Q'$. Under this bijection, we
%denote by $Q_M$ the subquiver of $Q$ corresponding to the
%indecomposable $M$. Since each representation of
%$RA_n$ can be identified as a representation of an orientation of a
%type-A Dynkin diagram, and these algebras are representation directed, we obtain the following proposition.

\begin{proposition}\label{cor: RA_n schur}
Each indecomposable module over $RA_n$ is a brick.
In particular, the canonical join complex of $\tors RA_n$ is isomorphic to the complex of hom-orthogonal indecomposable modules over $RA_n$.
\end{proposition}
\begin{proof}
Let $M$ be an indecomposable module over $RA_n$ and let $Q_M$ be the corresponding quiver.
An endomorphism $f:M\rightarrow M$ is a set of maps $f=(f_i)_{i\in Q_1}$ where $f_i:M(i)\rightarrow~M(i)$ and for every arrow $a:i\rightarrow j\in Q_1$, the composition $M(a)\circ f_i$ is equal to ${f_j\circ M(a)}$. Since $M(i)=k$ for all $i\in (Q_M)_0$, the map $M(a):k\rightarrow k$ is just a scalar multiplication. Therefore $f_i=f_j$ for each $i,j\in (Q_M)_0$, and hence $f$ is a scalar multiple of the identity map.
\end{proof}

%
%\begin{figure}[h]
%\centering
%\scalebox{.8}{
%\begin{tikzpicture}
%%\draw[help lines] (-7,-3) grid (7,3);
%%more help lines
%\draw [black] (0,-3) to (0,3);
%\draw [black] (-4,-3) to (-4,3);
%\draw [black] (4,-3) to (4,3);
%\draw[black] (-7,0) to (7,0);
%
%%top row
%%first
%\node at (-6.5,1) {$1$};
%\node at (-6,1.5) {$2$};
%\node at (-5.5,1) {$3$};
%\node at (-5,1.5) {$4$};
%
%%second
%\node at (-2.5,1.5) {$1$};
%\node at (-2,1) {$2$};
%\node at (-1.5,1.5) {$3$};
%\node at (-1,1) {$4$};
%
%%third
%\node at (2.5,1.5) {$4$};
%\node at (2,2) {$3$};
%\node at (1.5,1.5) {$2$};
%\node at (1,1) {$1$};
%
%%fourth
%\node at (6.5,1.5) {$4$};
%\node at (6,1) {$3$};
%\node at (5.5,1.5) {$2$};
%\node at (5,2) {$1$};
%
%%bottom row
%%first
%\node at (-6.5,-1.5) {$1$};
%\node at (-6,-1) {$2$};
%\node at (-5.5,-1.5) {$3$};
%\node at (-5,-2) {$4$};
%
%%second
%\node at (-2.5,-1.5) {$1$};
%\node at (-2,-2) {$2$};
%\node at (-1.5,-1.5) {$3$};
%\node at (-1,-1) {$4$};
%
%%third
%\node at (2.5,-2.5) {$4$};
%\node at (2,-2) {$3$};
%\node at (1.5,-1.5) {$2$};
%\node at (1,-1) {$1$};
%
%%fourth
%\node at (6.5,-1) {$4$};
%\node at (6,-1.5) {$3$};
%\node at (5.5,-2) {$2$};
%\node at (5,-2.5) {$1$};
%
%\end{tikzpicture}}
%\caption{A visualization of the indecomposable modules with full
%  support over $RA_4$, where one vertex $i$ is above its neighbor $j$ if the
%  arrow $i\rightarrow j$ acts non-trivially on $M$.}
%\label{fig: quivers for A4}
%\end{figure}

From now on, we abuse notation and refer to $\Gamma(\tors RA_n)$ as the complex of hom-orthogonal indecomposable modules over $RA_n$.
(Although, more precisely, $\Gamma(\tors RA_n)$ is a simplicial complex on the set of join-irreducible torsion classes $\filt(\Gen(M))$ in $\tors RA_n$, not the set indecomposable modules.)

We close this section with a technical lemma that will be useful in Section~\ref{sec:shards}.
\begin{lemma}\label{lem: ext}
Suppose that $M$ is an indecomposable module over $RA_n$, and
$S$ is an interval in $[n]$ containing $\supp(M)$. Then:
\begin{enumerate}
\item $M$ is a submodule of some indecomposable $M'$ with
  $\supp(M')=S$ and
\item there is an indecomposable $M''$ with $\supp(M'')=S$, of which
  $M$ is a quotient.
\end{enumerate}
\end{lemma}
\begin{proof}
We prove only the first statement, since the second is similar by
Proposition \ref{cor: factors and subs}. Let $Q_M$ be the quiver
associated with $M$, and write $[p,q]=\supp(M)$. Let $Q_{M'}$ be any
quiver with support equal to $S$ satisfying the following:
the orientation of $Q_{M'}$ on the interval $[p,q]$ coincides with that of $Q_M$, and
$Q_{M'}$ contains the the arrows $a_{p-1}$ if $p-1\in S$ and $a_{q}^*$ if $q+1\in S$.
Since $Q_M$ is a connected successor closed subquiver of $Q_{M'}$, we
obtain the desired result.
\end{proof}

%The weak order on any finite Coxeter group is a finite semidistributive lattice \cite[Lemma~9]{weak order}.
%Thus, each permutation has a canonical join representation.

%The \newword{transitive closure} of $S$, denoted $\tran(S)$, is the set of all pair $(i,k)$ satisfying:
%There is a finite subset $\{(i_0, j_0), (j_0, j_1), \ldots, (j_r, k)\}\subseteq S$.

\subsection{Noncrossing arc diagrams and canonical join representations}
\label{sec:arc-diagrams}
In this section, we construct a model for the canonical join complex of $\tors RA_n$ called the noncrossing arc complex.
The noncrossing arc complex was first defined in \cite{arcs} where it was used to study certain aspects of the symmetric group.
(We will make use of this connection in the following section.)
Informally, the noncrossing arc complex is a simplicial complex whose faces are collections of non-intersecting curves called arcs.
We will see that the ``noncrossing'' criteria that defines such a face also encodes the hom-orthogonality of indecomposable modules in $\module RA_n$.

%For further background, we direct the reader to~\cite{arcs}.
A \newword{noncrossing arc diagram} on $n+1$ nodes consists of a vertical column of nodes, labeled $0,\ldots, n$ in increasing order from bottom to top, together with a (possibly empty) collection of curves called arcs.
Each arc $\alpha$ has two endpoints, and travels monotonically up from its bottom endpoint $b(\alpha)$ to its top endpoint~$t(\alpha)$.
For each node in between, $\alpha$ passes either to the left or to the right.
Each pair of arcs $\alpha$ and $\beta$ in a diagram satisfies two compatibility conditions:
\begin{enumerate}[label={(C\arabic*)}]
\item\label{C1} $\alpha$ and $\beta$ do not share a bottom endpoint or a top endpoint;
%\item\label{C2} $\alpha$ and $\beta$ do not share a common top endpoint;
\item\label{C2} $\alpha$ and $\beta$ do not cross in their interiors.
\end{enumerate}
%Noncrossing diagrams were first defined by Reading in \cite{arcs}, where it was shown that the set of noncrossing diagrams on $n+1$ vertices are in bijection with the set of permutations in $A_n$
\begin{figure}[h]
  \centering
   \scalebox{1}{ \includegraphics{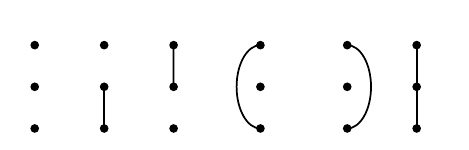}}
     \caption{The noncrossing arc diagrams on $3$ nodes.}
        \label{S3 arcs}
\end{figure}
Each arc is considered only up to combinatorial
equivalence.
That is, each arc $\alpha$ is characterized by its endpoints and which
side each node the arc passes (either left or right) as it travels from
$b(\alpha)$ up to $t(\alpha)$. 
Furthermore, a collection of arcs is drawn so as to have the smallest
number of intersections. 
The \newword{support of an arc} $\alpha$, written $\supp(\alpha)$, is the set $[b(\alpha), t(\alpha)]$.
We write $\suppo(\alpha)$ for the set $(b(\alpha), t(\alpha))$.
We say that $\alpha$ has full support if $\supp(\alpha) = [0,n]$.
We say that $\beta$ is a \newword{subarc} of $\alpha$ if both of the following conditions are satisfied:
\begin{itemize}
\item $\supp(\beta)\subseteq \supp(\alpha)$;
\item $\alpha$ and $\beta$ pass on the same side of each node in $\suppo(\beta)$.
\end{itemize}

 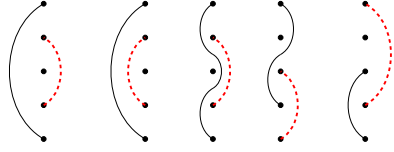
\begin{figure}[h]
 \centering
\scalebox{.45}{
\begin{tikzpicture}
%\draw[help lines] (-4,-2) grid (4,2);
%\draw[black] (1,-2) to (1,2);
%\draw[black] (-1.5,-2) to (-1.5,2);
%\draw[fill] (0,0) circle [radius=.04];

%diagram 1
\draw[fill] (-5,1) circle [radius=.07];
\draw[fill] (-5,-2) circle [radius=.07];
\draw[fill] (-5,0) circle [radius=.07];
\draw[fill] (-5,2) circle [radius=.07];
\draw[fill] (-5,-1) circle [radius=.07];
\draw [black, thick] (-5,-2) to [out= 150, in =-150 ] (-5,2);
\draw [red, ultra thick, dashed] (-5,-1) to [out= 30, in =-30 ] (-5,1);

%diagram 2
\draw[fill] (-2,1) circle [radius=.07];
\draw[fill] (-2,-2) circle [radius=.07];
\draw[fill] (-2,0) circle [radius=.07];
\draw[fill] (-2,2) circle [radius=.07];
\draw[fill] (-2,-1) circle [radius=.07];
\draw [black, thick] (-2,-2) to [out= 150, in =-150 ] (-2,2);
\draw [red, ultra thick, dashed] (-2,-1) to [out= 150, in =-150 ] (-2,1);

%diagram 3
\draw[fill] (0,1) circle [radius=.07];
\draw[fill] (0,-2) circle [radius=.07];
\draw[fill] (0,0) circle [radius=.07];
\draw[fill] (0,2) circle [radius=.07];
\draw[fill] (0,-1) circle [radius=.07];

\draw [black, thick] (0,-2) to [out= 150, in =210 ] (0,-.5);
\draw [black, thick] (0,-.5) to [out = 30, in = -30] (0,.5);
\draw [black, thick] (0,.5) to [out= 150, in =-150 ] (0,2);
\draw [red, ultra thick, dashed] (0,-1) to [out= 30, in =-30 ] (0,1);
%\draw [black, thick] (-2,1.5) to [out= 30, in =-30 ] (-2,2);

%diagram 4
\draw[fill] (2,1) circle [radius=.07];
\draw[fill] (2,-2) circle [radius=.07];
\draw[fill] (2,0) circle [radius=.07];
\draw[fill] (2,2) circle [radius=.07];
\draw[fill] (2,-1) circle [radius=.07];

\draw [black, thick] (2,-1) to [out= 150, in =210 ] (2,.5);
\draw [black, thick] (2, .5) to [out= 30, in =-30 ] (2,2);
\draw [red, ultra thick, dashed] (2,-2) to [out= 30, in =-30 ] (2,0);

%%%diagram 5
\draw[fill] (4.5,1) circle [radius=.07];
\draw[fill] (4.5,-2) circle [radius=.07];
\draw[fill] (4.5,0) circle [radius=.07];
\draw[fill] (4.5,2) circle [radius=.07];
\draw[fill] (4.5,-1) circle [radius=.07];

%arcs
\draw [black, thick] (4.5,-2) to [out= 150, in =-150 ] (4.5,0);
\draw [red, ultra thick, dashed] (4.5,-1) to [out= 30, in =-30 ] (4.5,2);

\end{tikzpicture}
}
\caption{Some pairs of compatible arcs.}
\label{fig: left and right position}
\end{figure}
A set of arcs are \newword{compatible} if there is a noncrossing arc diagram that contains them.
We define the \newword{noncrossing arc complex} on $n+1$ nodes to be the simplicial complex whose vertex set is the set of arcs and whose face set is the collection of all sets of compatible arcs.
(We view each collection of compatible arcs as a noncrossing arc diagram.
When we refer to ``the set of arcs'' we mean ``the set of noncrossing arc diagrams, each of which contains precisely one arc''.)
The next proposition is a combination of \cite[Proposition~3.2 and Corollary~3.6]{arcs}.
Recall that a simplicial complex is \newword{flag} if its minimal non-faces have size equal to 2.
Equivalently, a subset~$F$ of vertices is a face if and only if each pair of vertices is a face.
\begin{proposition}\label{arc flag}
A collection of arcs can be drawn together in a noncrossing arc diagram if and only if each pair of arcs is compatible.
Thus, the noncrossing arc complex is flag.
\end{proposition}

Our goal is to prove:

\begin{theorem}\label{thm: cjc isomorphism}
The canonical join complex of the lattice $\tors RA_n $ is isomorphic to the noncrossing arc complex on $n+1$ nodes.
\end{theorem}

We begin the proof of Theorem~\ref{thm: cjc isomorphism} by mapping vertices to vertices.
More precisely, we define a bijection $\sigma$ from the set of
indecomposable modules over $RA_n$ to the set of arcs on $n+1$ nodes.

For an arc $\alpha$ with support $[p-1,q]$ we define:
\begin{align*}
R(\alpha)&=\{i\in [p,q-1]:\, \text{$\alpha$ passes on the right side of $i$}\};\\
L(\alpha)&=\{i\in [p,q-1]:\, \text{$\alpha$ passes on the left side of $i$}\}.
\end{align*}

For an indecomposable $RA_n$ module with support $[p,q] \subseteq [n]$,
we define:
\begin{align*}
R(M)&=\{i\in [p,q-1]:\, \text{$a_{i}$ acts nontrivially on $M$}\};\\
L(M)&= \{i\in [p,q-1]:\, \text{$a_{i}^*$ acts nontrivially on $M$}\}.
\end{align*}
%Define $\sigma(M)$ to be the
%arc with support $[p-1,q]$ which passes to the right (resp. left) of the node $i\in
%(p-1,q)$ if $a_i$ (resp. $a_i^*$) acts non-trivially on $M$.
Just as an arc is determied by its endpoints and the binary Left-Right
data, so too is an indecomposable module over $RA_n$ determined by the
binary data of the action of its Lowering-Raising arrows ($a_i^*$ and
$a_i$, respectively). Therefore, we have the following:

\begin{proposition}\label{prop: arc bij}
Let $\sigma$ be the map which sends an indecomposable~$RA_n$ module~$M$ with support $[p,q]$ to the arc $\sigma(M)=\alpha$ satisfying
\begin{itemize}
\item $b(\alpha)=p-1$ and $t(\alpha)=q$;
\item $L(\alpha)=L(M)$;
\item $R(\alpha)=R(M).$
\end{itemize}
Then $\sigma$ is a bijection from the set of indecomposable modules
over $RA_n$ the set of arcs on $n+1$ nodes. (See Figure \ref{fig:
  arcmap example}.)
\end{proposition}

 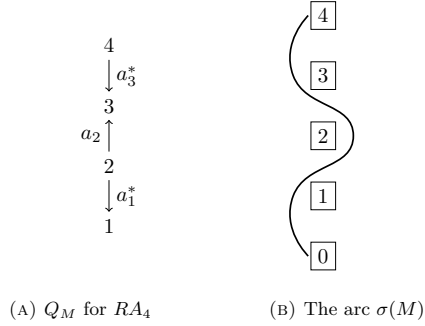
\begin{figure}[h]
 \centering
 \scalebox{.8}{
\begin{subfigure}[b]{0.4\linewidth}
 \begin{tikzpicture}
   \node (A) at (0,0.5) {1};
   \node (B) at (0,1.5) {2};
   \node (C) at (0,2.5) {3};
   \node (D) at (0,3.5) {4};
   \draw[->] (B)--(A) node[midway,right] {$a_1^*$};
   \draw[->] (B)--(C) node[midway,left] {$a_2$};
   \draw[->] (D)--(C) node[midway,right] {$a_3^*$};
   \draw (-3,-0.5);
 \end{tikzpicture}
\caption{$Q_M$ for  $RA_4$}
\end{subfigure}\hfill
\begin{subfigure}[b]{0.3\textwidth}
 \begin{tikzpicture}
   \foreach \n in {0,1,2,3,4}{
     \node[draw] at (0,\n) {\n};
   }
   \draw[thick] plot [smooth,tension=1] coordinates
   {(-0.25,0) (-0.5,1) (0.5,2) (-0.5,3) (-0.25,4)};
   \draw (-1.5,-0.5);
 \end{tikzpicture}
\caption{The arc $\sigma(M)$}
\end{subfigure}}
\caption{Visualization of $\sigma(M)$  and $Q_M$ for a module over $RA_4$. }
\label{fig: arcmap example}
\end{figure}

For the remainder of the section, we let $M$ and $M'$ be indecomposable modules, and write $\alpha$ for $\sigma(M)$ and $\alpha'$ for $\sigma(M')$.
We wish is to reinterpret the hom-orthogonality of $M$ and $M'$ in terms of certain subarcs of $\alpha$ and $\alpha'$.
 Recall that a quiver $Q'$  is called a \newword{predecessor closed subquiver of $Q$} if
$i\rightarrow j$ with $j\in Q'$ implies $i\in Q'$. \newword{Successor closed
subquivers} are defined similarly.
The following result is well-known (and an easy exercise).
See~\cite[Section~2]{Crawley-Boevey}.
\begin{proposition}\label{cor: factors and subs}
Suppose that $M$ and $M'$ are indecomposable modules over $RA_n$
and let $Q_M$ and $Q_{M'}$ be the corresponding quivers.
Then:
\begin{enumerate}
\item $M'$ is a quotient of $M$ if and only if $Q_{M'}$ is a connected predecessor closed subquiver of $Q_M$.
\item $M'$ is a submodule of $M$ if and only if $Q_{M'}$ is a connected successor closed subquiver of $Q_M$.
\end{enumerate}
\end{proposition}

We define an analogous notion for arcs.
We say that $\beta$ is a \newword{predecessor closed subarc} of $\alpha$ if $\beta$ is a subarc of $\alpha$, and $\alpha$ does not pass to the right of $b(\beta)$, nor to the left of $t(\beta)$.
Similarly, $\beta$ is a \newword{successor closed subarc} if $\alpha$ does not pass to the left of $b(\beta)$ nor to the right of $t(\beta)$.
Observe that each predecessor closed subarc of $\alpha$ corresponds (via the map $\sigma$) to an indecomposable quotient of~$M$--that is, a \textit{predecessor closed} subquiver of $Q_M$.
(The analogous statement holds for sucessor closed subarcs of $\alpha$.)
The next result is from Crawley-Boevey \cite[Section~2]{Crawley-Boevey}, rephrased for our context.

\begin{proposition}\label{prop:homs-between-shards}
Let $\alpha$ and $\alpha'$ be arcs on $n+1$ nodes, and write $M$ and $M'$ for the corresponding indecomposable modules over $RA_n$.
Then the vector space $\Hom_{RA_n} (M, M')$ has dimension equal to the number of predecessor closed subarcs of $\alpha$ which are also successor closed subarcs of $\alpha'$.
\end{proposition}

%By Theorem~\ref{main: canonical join complex}, the torsion classes $\calT$ and $\calT'$ are compatible if and only if there is no arc $\beta$ is that is a predecessor closed subarc of one arc, either $\alpha$ or $\alpha'$, and also a successor closed subarc of the other arc.
%Thus, to prove Theorem~\ref{thm: iso to arc complex}, we argue that $\alpha$ and $\alpha'$ are compatible if and only if there exists no arc $\beta$ that is a predecessor closed subarc of one, either $\alpha$ or $\alpha'$ and successor closed subarc of the other.
We argue that $\alpha$ and $\alpha'$ are compatible if and only if
there exists no arc~$\beta$ that is a predecessor closed subarc of one
which is simultaneously a successor closed subarc of the other.
\begin{figure}[h]
\scalebox{.8}{
\begin{tikzpicture}
%\draw[help lines] (-1,-3) grid (1,3);
%\draw[fill] (0,0) circle [radius=.04];
%nodes of the diagram
\draw[fill] (0,1) circle [radius=.02];
\draw[fill] (0,.5) circle [radius=.02];
%\draw[fill] (0,.25) circle [radius=.02];
\draw[fill] (0,.75) circle [radius=.02];
\draw[fill] (0,1.25) circle [radius=.02];
\draw[fill] (0,1.75) circle [radius=.02];
\draw[fill] (0,1.5) circle [radius=.02];
\draw[fill] (0,-2) circle [radius=.07];
\draw[fill] (0,2) circle [radius=.07];
\draw[fill] (0,-1) circle [radius=.02];
\draw[fill] (0,-.5) circle [radius=.02];
\draw[fill] (0,-1.25) circle [radius=.02];
\draw[fill] (0,-1.75) circle [radius=.02];
\draw[fill] (0,-1.5) circle [radius=.02];
%\draw[fill] (0,-.25) circle [radius=.02];
\draw[fill] (0,-.75) circle [radius=.02];

%arcs
\draw [red, ultra thick, dashed] (1,-2.5) to [out= 90, in =-20 ] (0,0);
\draw [red, ultra thick, dashed] (0,0) to [out= 160, in =-90 ] (-1,2.5);
\draw [black, thick] (-1,-2.5) to [out= 90, in =200 ] (0,0);
\draw [black, thick] (0,0) to [out = 20, in = -90] (1,2.5);

%labels
\node at (-1.5, -2) {$\alpha$};
\node at (0,-2.5) {$b(\beta)$};
\node at (0, 2.5) {$t(\beta)$};
\node at (1.5, -2) {$\alpha'$};
\end{tikzpicture}
}

\caption{Suppose that $\beta$ is a predecessor closed subarc of $\alpha$ and a successor closed subarc of $\alpha'$.
Then the endpoints of~$\beta$ lie between these two arcs.
The two arcs switch from left side to right side (and vice versa) as they travel from $b(\beta)$ to~$t(\beta)$. }
\label{fig: subarc}
\end{figure}
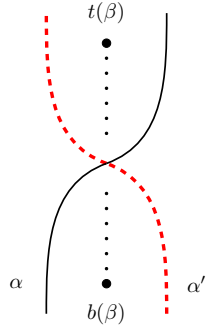
The following two lemmas give one direction of that argument.
Figure~\ref{fig: subarc} may help build some intuition.

At times it will be convenient to consider the relative position of a pair of ``overlapping'' arcs.

We say that $\alpha$ and $\alpha'$ \newword{overlap} if the set $$\left(\supp(\alpha)\cap\suppo(\alpha')\right)\cup \left(\supp(\alpha')\cap \suppo(\alpha)\right)$$ is nonempty.
For two arcs $\alpha$ and $\alpha'$, we say that $\alpha$ is left of $\alpha'$ if both of the following hold:
\begin{itemize}
\item $\left(R(\alpha) \cup \{t(\alpha),b(\alpha)\}\right) \cap \left(\suppo(\alpha') \right)\subseteq R(\alpha')$
\item $\{t(\alpha'), b(\alpha')\} \cap \left(\suppo(\alpha)\right) \subseteq L(\alpha)$.
\end{itemize}
%When $\alpha$ and $\alpha'$ are compatible, we say that $\alpha$ is
%left of $\alpha'$ if, whenever $\alpha$ passes to the right of a node, either $\alpha'$ also passes to the right of the same node or $\alpha'$ ends at that node.
For example, in each diagram in Figure~\ref{fig: left and right position}, the arc $\alpha$ (solid) is \textit{left} of the arc $\alpha'$ (dashed).

\begin{lemma}\label{special subarc shared endpoint}
Suppose that $\alpha$ and $\alpha'$ are distinct arcs that share a bottom endpoint or a top endpoint.
Then there exists an arc $\beta$ satisfying:
$\beta$ is a predecessor closed subarc of one arc, either $\alpha$ or $\alpha'$, and a successor closed subarc of the other arc.
\end{lemma}
\begin{proof}
By symmetry, we assme that $\alpha$ and $\alpha'$ share a bottom node, and, without loss of generality, this bottom endpoint is equal to 0.
Let $q$ be the smallest number such that either of two conditions below is satisfied:
\begin{itemize}
\item $\alpha$ and $\alpha'$ pass on opposite sides of $q$;
\item $q= \min\{t(\alpha), t(\alpha')\}$.
\end{itemize}

Let $\beta$ be the arc with endpoints $b(\beta)=0$ and $t(\beta)=q$ such that $\beta$ is a subarc of $\alpha$.
(Note, if $q= t(\alpha)$ then $\beta$ coincides with $\alpha$.
If $q\ne t(\alpha)$, we can visualize $\beta$ by cutting $\alpha$ where it passes beside the node $q$, and anchoring the resulting segment at $q$.)
Since $\alpha$ and $\alpha'$ pass on the same side of each node in the
set $[0,q-1]$, we conclude that $\beta$ is also a subarc of $\alpha'$.

At least one of the two arcs, $\alpha$ or $\alpha'$, passes $q$.
By symmetry, assume that this is the arc $\alpha$, and assume that $\alpha$ passes on the right side of $q$.
Then $\beta$ is a predecessor closed subarc of $\alpha$ and also a
successor closed subarc of $\alpha'$.
The other choices yield similar results. 
%If $\alpha$ passes to the left, then $\beta$ is a successor closed subarc of $\alpha$ and also a predecessor closed subarc of $\alpha'$.
\end{proof}

\begin{lemma}\label{crossing}
Suppose that $\alpha$ and $\alpha'$ are distinct arcs that neither
share a bottom nor top endpoint.
If $\alpha$ and $\alpha'$ intersect in
their interiors, then there exists an arc $\beta$ which is a predecessor closed subarc of one arc and a successor closed subarc of the other.
\end{lemma}
\begin{proof}
First, we consider the case in which one of the two arcs is a subarc of the other.
By symmetry, we assume that $\alpha'$ is a subarc of $\alpha$.
Observe that $b(\alpha')$ and $t(\alpha')$ must be on opposite sides of $\alpha$.
(Otherwise, $\alpha$ and $\alpha'$ can be drawn so they do not cross.)
Therefore, by definition, either $\alpha'$ is a predecessor closed or
successor closed subarc of~$\alpha$.

Next, suppose that $\alpha'$ and $\alpha$ pass on the same side of each node in the set $\suppo(\alpha)\cap\suppo(\alpha')$, but neither arc is a subarc of the other.
By symmetry, assume that $b(\alpha')<b(\alpha) < t(\alpha')<t(\alpha)$ and that $\alpha'$ passes to the right of $b(\alpha)$.
%(Equivalently, $b(\alpha)$ is on the right side of $\alpha'$.)
Observe that $\alpha$ also passes to the right of $t(\alpha')$.
Otherwise, the arcs can be drawn so that $\alpha'$ lies strictly to the left of $\alpha$, and the two arcs never cross.
See Figure~\ref{fig: crossing}.
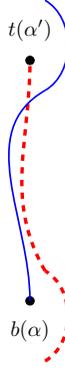
\begin{figure}[h]
\scalebox{.8}{
\begin{tikzpicture}
%\draw[help lines] (-1,-3) grid (1,3);
%\draw[fill] (0,0) circle [radius=.04];
%nodes of the diagram

%%\draw[fill] (0,1) circle [radius=.02];
%%\draw[fill] (0,.5) circle [radius=.02];
%%\draw[fill] (0,.25) circle [radius=.02];
%%\draw[fill] (0,.75) circle [radius=.02];
%%\draw[fill] (0,1.25) circle [radius=.02];
%\draw[fill] (0,1.75) circle [radius=.02];
%\draw[fill] (0,1.5) circle [radius=.02];
\draw[fill] (0,-2) circle [radius=.07];
\draw[fill] (0,2) circle [radius=.07];
%\draw[fill] (0,-1) circle [radius=.02];
%\draw[fill] (0,-.5) circle [radius=.02];
%\draw[fill] (0,-1.25) circle [radius=.02];
%\draw[fill] (0,-1.75) circle [radius=.02];
%\draw[fill] (0,-1.5) circle [radius=.02];
%%\draw[fill] (0,-.25) circle [radius=.02];
%\draw[fill] (0,-.75) circle [radius=.02];

%arcs
\draw [red, ultra thick, dashed] (.25,-3) to [out= 30, in =-30 ] (.25,-1.5);
\draw [red, ultra thick, dashed] (.25,-1.5) to [out= 120, in =-90 ] (0,2);

\draw [blue, thick,] (.25,3) to [out= -30, in = 30 ] (.25,1.5);
\draw [blue, thick,] (.25,1.5) to [out= 210, in =90 ] (0,-2);

%labels
%\node at (1.5, 2) {$\alpha$};
\node at (0,-2.5) {$b(\alpha)$};
\node at (0, 2.5) {$t(\alpha')$};
%\node at (1.5, -2) {$\alpha'$};
\end{tikzpicture}
}

\caption{A demonstration of the proof of Lemma~\ref{crossing}.
The arc $\alpha$ is shown in blue, and $\alpha'$ is shown in dashed red.}
\label{fig: crossing}
\end{figure}Let~$\beta$ be the subarc of $\alpha$ with bottom endpoint $b(\alpha)$ and top endpoint $t(\alpha')$.
Since $\alpha$ and $\alpha'$ pass on the same side of each node where they overlap, $\beta$ is also a subarc of $\alpha'$.
We conclude that $\beta$ is a predecessor closed subarc of $\alpha$ that is also a successor closed subarc of $\alpha'$.

Finally, we assume that $\alpha$ and $\alpha'$ pass on opposite sides
of some node $p$ belonging to
$\suppo(\alpha)\cap\suppo(\alpha')$. 
By symmetry, assume that $\alpha$ passes to the left side of $p$ and $\alpha'$ passes to the right.
Consider the set \[O=\left(\supp(\alpha)\cap\suppo(\alpha')\right)\cup \left(\supp(\alpha')\cap \suppo(\alpha)\right).\]
By way of contradiction, suppose for each $q\in O$, either $\alpha$ passes on the left side of $q$ or $\alpha'$ passes on the right of $q$.
Then for each $q\in O$, either $\alpha$ and $\alpha'$ pass on the same side of $q$, or $q$ lies between them, with $\alpha$ on the left and $\alpha'$ on the right.
We conclude that $\alpha$ is strictly to the left of $\alpha'$, and
the two arcs never cross, a contradiction.

Thus, there exists some $q\in O$ satisfying: $\alpha$ does not pass to
the left of $q$ nor does $\alpha'$ pass to the right. 
We choose $p$ and $q$ so that $|p-q|$ is minimal.
By symmetry, assume that $p<q$.
Let $\beta$ be the subarc of $\alpha$ with endpoints $p<q$.
The minimality of $|p-q|$ implies that $\beta$ is also a subarc of $\alpha'$.
We conclude that $\beta$ is a predecessor closed subarc of $\alpha$ that is also a successor closed subarc of $\alpha'$.
\end{proof}
Together, the previous two lemmas (and Proposition~\ref{prop:homs-between-shards}) imply that if $\alpha$ and $\alpha'$ are not compatible, then there exists some homomorphism between $M$ and $M'$.
The next lemma completes the proof of Theorem~\ref{thm: cjc isomorphism}.

\begin{lemma}\label{lem: subarc implies crossing}
Suppose that $\alpha$ and $\alpha'$ are compatible.
Then there exists no arc $\beta$ that is predecessor closed subarc of
one arc and a successor closed subarc of the other.
Thus, the modules $M$ and $M'$ corresponding to $\alpha$ and $\alpha'$
under the bijection $\sigma$ are hom-orthogonal. 
\end{lemma}

\begin{proof}
By way of contradiction, assume there exists an arc $\beta$ that is both a predecessor closed subarc of $\alpha$ and a successor closed subarc of $\alpha'$.
Because $\alpha$ and $\alpha'$ do not share a bottom endpoint, either $b(\alpha)\ne b(\beta)$ or $b(\alpha')\ne b(\beta)$.
By symmetry, assume that $b(\alpha)\ne b(\beta)$.
Thus, $\alpha$ passes strictly to the left of $b(\beta)$.
Since $\alpha'$ does not pass to the left of $b(\beta)$, we conclude that $\alpha$ lies strictly on the left side of $\alpha'$, wherever the two arcs overlap.

%Thus, any node $q$ in $\suppo(\alpha)\cap \suppo(\alpha')$ that is strictly on the left side of $\alpha$ is also strictly to the left of of $\alpha'$.
%(Equivalently, $\alpha'$ passes to the right of such a node $q$.)

For the moment, assume that $t(\beta)\in \suppo(\alpha)\cap \suppo(\alpha')$.
On the one hand, $\alpha$ passes strictly to the right of $t(\beta)$.
On the other hand, $\alpha'$ passes strictly to the left of $t(\beta)$, so that $\alpha$ and $\alpha'$ intersect in their interiors.
(See Figure~\ref{fig: subarc}.)
%We conclude that $\alpha'$ also passes strictly to the right of $t(\beta)$.
%This is a contradiction, as $\beta$ cannot be a successor closed subarc of $\alpha'$.
We obtain a similar contradiction if $t(\beta)=t(\alpha)$.

Finally, suppose that $t(\beta)= t(\alpha')$.
Since $\alpha$ and $\alpha'$ do not share a top endpoint and $\alpha$ is left of $\alpha'$, we conclude that $\alpha$ passes strictly to the left of $t(\alpha')$.
This contradicts the fact that $\beta$ is a predecessor closed subarc of~$\alpha$.
\end{proof}

\subsection{The weak order on $A_n$ and $\tors RA_n$}
\label{sec:shards}
In this section we prove that $\tors RA_n$ is isomorphic to the weak order on $A_n$.
We begin by reviewing the weak order and its connection to noncrossing arc diagrams.

Recall that the type-A Weyl group of rank $n$ is isomorphic to the symmetric group on $[n+1]$.
It will be convenient for us to realize $A_n$ as the symmetric group on $\{0,\ldots,n\}= [0,n]$.
For the remainder of the paper, we do not distinguish between the elements of $A_n$ and the permutations on $[0,n]$.
We write $w\in A_{n}$ in its one-line notation as $w= w_0\ldots w_{n}$ where $w_i=w(i)$.
An \newword{inversion} of $w$ is a pair $(w_i,w_j)$ with $w_i>w_j$ and $i<j$.
Each permutation is uniquely determined by its inversion set.

In the weak order, permutations are partially ordered by containment of their inversion sets.
In particular, $w\covers v$ if and only if $\inv(v) \subseteq \inv(w)$ and $\inv(w) \setminus \inv(v)$ has precisely one element.
This unique inversion is a descent for $w$.
(Recall that a \newword{descent} is an inversion $(p,q)$ such that $q=w_{i}$ and $p=w_{i+1}$.)
In this way, each descent of $w$ corresponds bijectively to a permutation that is covered by $w$.
It follows that $w$ is join-irreducible if and only if it has precisely one descent.

We describe a bijection $\delta$ from the set of permutations in $A_n$
to the set of noncrossing arc diagrams on $n+1$ nodes.
For each descent $(w_i,w_{i+1})$ of a permutation $w$, construct an arc
$\alpha$ satisfying: $t(\alpha)=w_i$, $b(\alpha)=w_{i+1}$, 
and $w_j \in R(\alpha)$ (resp. $w_j \in L(\alpha)$) whenever $w_j$ is
in the interval $(b(\alpha),t(\alpha))$ and $j>i+1$
(resp. $j<i$). 
The arc diagram $\delta(w)$ is then the union of these arcs. 
% Given $w=w_0\ldots w_{n}$, we plot each point $(i,w_i)$.
% We connect $(i, w_i)$ to $(i+1, w_i)$ with a straight line segment whenever $w_i>w_{i+1}$.
% (That is, whenever the pair $w_i$ and $w_{i+1}$ form a descent.)
% Finally, we move all of the points into a vertical line, bending the straight line segments so that they become the arcs in our diagram.
Since descents become arcs, it follows that $\delta$ restricts to a bijection from the set of join-irreducible permutations (on $[0,n]$) to the set of arcs (on $n+1$ nodes).
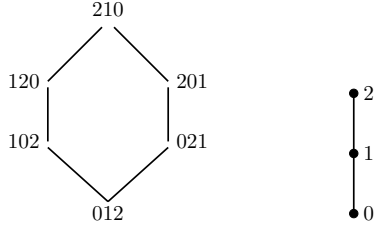
\begin{figure}[h]
  \centering
  \scalebox{.8}{
\begin{tikzpicture}
%\draw [black, thick] (0.1,1.45) -- (.68,2.1);
\draw [black, thick] (0,0) -- (-1,.9);
\draw [black, thick] (0,0) -- (1,.9);
\draw [black, thick] (1,1) -- (1,1.9);
\draw [black, thick] (-1,1) -- (-1,1.9);
\draw [black, thick] (1,2) -- (.1,2.9);
\draw [black, thick] (-1,2) -- (-.1,2.9);
\node at (0,-.2) {$012$};
\node at (1.4,1) {$021$};
\node at (-1.4,1) {$102$};
\node at (1.4,2) {$201$};
\node at (-1.4,2) {$120$};
\node at (0,3.2) {$210$};
\end{tikzpicture}
\qquad
\qquad
\qquad

\begin{tikzpicture}
%nodes
\draw[fill] (0,2) circle [radius=.07];
\draw[fill] (0,1) circle [radius = .07];
\draw[fill] (0,0) circle [radius=.07];

%arcs
\draw [black, thick] (0,0) to (0,1);
\draw [black, thick] (0,2) to (0,1);
%\draw [black, thick] (0,0) to [out = 20, in = -90] (1,2.5);

%labels
\node at (0.25,0) {$0$};
\node at (.25, 2) {$2$};
\node at (.25, 1) {$1$};
\end{tikzpicture}
}
 \caption{The weak order on the symmetric group $A_2$ and the noncrossing arc diagram corresponding to $210$.}
        \label{cayley}
\end{figure}
\begin{example}\label{delta example}
\normalfont
Consider the permutation $w= 210$, the top element in the weak order on $A_2$.
The noncrossing arc diagram $\delta(w)$ consists of two arcs, one connecting 0 to 1, and the second connecting 1 to 2.
(See Figure~\ref{cayley}.)
Each arc corresponds (via $\delta$) to a join-irreducible permutation.
In this example, the arc connecting 0 to 1 corresponds to $102$, and the arc which connects 1 to 2 corresponds to $021$.
Observe that that $210 = \Join \{102, 021\}$.
In fact, $\Join \{102, 021\}$ is the canonical join representation of $210$.
This fact is no coincidence.
\end{example}

In general, the weak order on $A_n$ is a lattice in which each permutation has a canonical join representation \cite{weak order}.
The next theorem is a combination of \cite[Theorem~3.1 and
Corollary~3.4]{arcs}.

\begin{theorem}\label{prop: A_n CJC}
The bijection~$\delta$ induces an isomorphism from the canonical join complex of the weak order on $A_n$ to the noncrossing arc complex on $n+1$ nodes.
\end{theorem}

We immediately obtain the following corollary.
\begin{corollary}\label{cor: RA_n cjc}
The canonical join complex of $RA_n$ is isomorphic to the canonical join complex of the weak order on $A_n$.
\end{corollary}

Let us a consider the following composition:
\[\Gamma(\tors RA_n) \xrightarrow{\sigma} \{\text{Noncrossing arc diagrams on $n+1$ nodes}\} \xrightarrow{\delta^{-1}} \Gamma(A_n)\]
%We claim that $\delta^{-1} \circ \sigma$ induces a bijection from the set of torsion classes in $\tors RA_n$ to the set of permutations in $A_n$.
Recall that the map $\sigma$ sends a collection of hom-orthogonal modules $\calE$ to a collection of compatible arcs $\sigma(\calE)= A$.
By Theorem~\ref{prop: A_n CJC}, $\delta^{-1}$ sends this collection of compatible arcs to the permutation $w = \Join \{\delta^{-1}(\alpha):\, \alpha\in A\}$, where this join $\Join \{\delta^{-1}(\alpha):\, \alpha\in A\}$ is the canonical join representation of $w$.
Define a map ${\phi: \tors RA_n \to A_n}$ as follows:
\[\phi: \Join \{\filt(\Gen(M)):\, M\in \calE\} \mapsto \Join \{\delta^{-1}(\sigma(M)):\, M\in \calE\}.\]
%Because each torsion class in $\tors RA_n$ has a canonical join representation (by Corollary~\ref{cor: torsion labeling}) and because the canonical join representation is unique, we conclude that $\phi$ is a bijection.
Because each torsion class in $\tors RA_n$ and each permutation in $A_n$ has a canonical join representation (see Corollary~\ref{cor: torsion labeling}), and because the canonical join representation is unique, we conclude that $\phi$ is a bijection.

\begin{example}\label{phi example}
\normalfont
Consider the join of torsion classes $\calT = \Join \{\{1\}, \{2\}\}$  in $\tors RA_2$.
(We write $\{i\}$ for the torsion class consisting of the simple module at vertex $i$.)
Observe that $\calT$ is the equal to the entire module category over $RA_2$.
Since the simple modules $1$ and $2$ are hom-orthogonal, the join $\Join \{\{1\}, \{2\}\} $ is the canonical join representation of $\calT$.

The map $\sigma$ sends $\{i\}$ to the arc with endpoints $i-1$ and $i$.
Thus, $\sigma(\calT)$ is the noncrossing arc diagram with two arcs: one arc that connects 0 to 1; and another arc that connects~1 to~2.
Recall from Example~\ref{delta example} that the permutation associated to this diagram is $210$.
Thus, $\phi(\calT)= 210$.
\end{example}

Our goal is to show that $\phi$ is actually a lattice isomorphism.
To that end, we now give an alternative description of $\phi$ in terms of inversion sets.
Recall that permutations in $A_n$ are ordered by containment of inversion sets.
Each inversion set $I = \inv(w)$ is \newword{transitively closed}, meaning that whenever $\{(p,q), (q,r)\}$ is a subset of $I$ then $(p,r)\in I$.
Consider the set $S$ of all pairs $(p,q)$ such that $p<q$ and $p,q\in [0,n]$, and write $\Omega$ for the power set of $S$.
The \newword{transitive closure of $I$} is the unique smallest (under containment) set $\tran(I)$ in $\Omega$ that is transitively closed and contains $I$.
Just as we compute the join of a set of torsion classes by taking their filtration closure, we compute the join of a set of permutations by taking their transitive closure.
More precisely:
\begin{proposition}\label{join of perms}
Suppose that $U$ is a collection of permutations in $A_n$.
The inversion set of the permutation $w=\Join U$ is equal to the transitive closure of the set $I= \{\inv(u):\, u\in U\}$.
\end{proposition}

Building on this analogy, let us define an ``inversion set'' for a torsion class of~$\module RA_n$.
For each indecomposable module $M$ with support $[p,r]$ we associate to $M$ the pair $\inv(M) = (p-1,r)$.
We define the \newword{inversion set of a torsion class} $\calT$ to be the set of all pairs $\inv(M)$ such that $M\in \ind(\calT)$.

\begin{lemma}\label{filt = tran}
Suppose that $S$ is a collection of indecomposable modules over $RA_n$.
Then $\inv(\filt(S))$ is equal to the transitive closure of $\inv(S)$.
\end{lemma}
\begin{proof}
First, suppose that $(p-1,r)\in \inv(\filt(S))$, and let $M$ be an indecomposable module in $\filt^l(S)$ with $\supp(M) = [p, r]$.
Without loss of generality, we assume that $M$ is not simple.
Thus, $p<r$.
We prove that $(p-1,r)\in \tran(\inv(S))$ by induction on $l$.
%The base case is trivial.

Let $M=M_{l}\supsetneq M_{l-1} \supsetneq \dotsc \supsetneq M_0 =0$ be an $S$-filtration of $M$, and consider the short exact sequence:
\[0\to M_{l-1} \to M\to M/M_{l-1}\to 0\]
Since ${\dim(M_{l-1}(i)) + \dim(M/M_{l-1}(i)) = \dim(M(i))}$ for each
vertex $i\in Q_0$, we conclude that $\supp(M_{l-1})=[p,q-1]$ and $\supp(M/M_{l-1}) =[q,r]$, for some $q\in (p,r)$.
Because $M_{l-1}$ belongs to $\filt^{l-1}(S)$, our inductive hypothesis implies that $(p-1,q-1)\in\tran(\inv(S))$.
Also $M/M_{l-1}\in S$, so we have $(q-1,r)\in \inv(S)$.
Thus, $(p-1,r)\in  \tran(\inv(S))$.
The other containment follows immediately from Lemma~\ref{lem:
  ext}.
%Next, suppose that $(i,k)\in \tran(\inv(S))$.
%Thus, there exists a finite subset of $S$ with the form $\{(i, j_0), (j_0, j_1), \ldots, (j_r, k)\}$.
%We prove that $(i,k) \in \inv(\filt(S))$ by induction on $r$.
%By induction, there exists $M', M'' \in \filt(S)$ with $\supp(M') = [i,j_{r-1}-1]$ and $\supp(M'') = [j_{r-1},k-1]$.
%Lemma~\ref{lem: ext} implies that there exists $M\in \filt(S)$ with support equal to $[i,k-1]$.
%The statement follows.
\end{proof}

\begin{proposition}\label{prop: inversion set}
Suppose that $\calT$ is a torsion class in $\tors RA_n$.
Then inversion set of $\calT$ is equal to $\inv(\phi(\calT))$.
\end{proposition}
 \begin{proof}
 First suppose that $\calT$ is join-irreducible, so that $\calT = \filt(\Gen(M))$, where $M$ is an indecomposable module over $RA_n$.
Without loss of generality, we assume that $M$ has full support.
As in Proposition~\ref{prop: arc bij}, let $\alpha= \sigma(M)$, where
$\alpha$ is defined by the condition that $L(\alpha)= L(M)$ and $R(\alpha)= R(M)$.
(Recall from Section~\ref{sec:arc-diagrams} that $L(M)$ is the set of
indices of the arrows of type $a_i^*$ in $Q_M$ while $L(\alpha)$ is
the set of nodes $i$ such that $\alpha$ passes on the left side of
$i$.)
Write $w$ for the join-irreducible permutation $\delta^{-1}(\alpha)$.
On the one hand, a pair $(p-1,r)$ is an inversion of $w$ if and only if $p-1\in L(\alpha) \cup \{0\}$ and $r\in R(\alpha)\cup \{n\}$.
On the other hand, $(p-1,r)$ is $\inv(N)$ for some factor $N$ of $M$ if and only if $p-1\in L(M) \cup \{0\}$ and $r\in R(M)\cup \{n\}$.
%an inversion of a factor of $M$ if and only if $p-1\in L(M) \cup \{0\}$ and $r\in R(M)\cup \{n\}$.
(Equivalently, each factor of $M$ has support $[p,r]$ such that $p-1\in L(M)\cup \{0\}$ and $r\in R(M)\cup \{n\}$.)
We conclude that $\inv(\Gen(M))$ is equal to $\inv(\phi(\calT))$.
Since the inversion set of the permutation $\phi(\calT)$ is already transitively closed, we conclude that $\inv(\calT)$ is equal to $\inv(\phi(\calT))$.

Now assume that $\calT$ is not join-irreducible.
Write $\Join \{\filt(\Gen(M)):\, M\in \calE\}$ for its canonical join representation, and write $A$ for the set $\delta^{-1}(\sigma(\calE))$.
Recall that $\phi(\calT) = \Join \{w:\, w\in A\}$, and this join is computed by taking the transitive closure of the set $\{\inv(w):\,w\in A\}$.
The join $\Join \{\filt(\Gen(M)):\,M\in \calE\}$ is computed by taking the filtration closure of the set $\{\Gen(M):\,M\in \calE\}$.
By Lemma~\ref{filt = tran}, $\inv(\calT)$ is equal to $\inv(\phi(\calT))$.
 \end{proof}

We are now prepared to prove the main result of this section.

\begin{theorem}\label{inversion and cover}
Let $\calT'$ and $\calT$ be torsion classes in $\tors RA_n$.
Suppose that $\calT'$ is equal to $\filt(\calT \cup\{M\})$, where $M$ is a minimal extending module for $\calT$.
Then $${\inv(\calT') \setminus \inv(\calT)} = \{\inv(M)\}.$$
In particular, $\tors RA_n$ is isomorphic to the weak order on $A_n$.
\end{theorem}
\begin{proof}
Write $(q,r)$ for the inversion $\inv(M)$.
Proposition~\ref{filt = tran} says that $\inv(\calT') = \tran(\inv(\calT) \cup \{(q,r)\})$.
If $(q,r)\in \inv(\calT)$ then $\inv(\calT)$ is equal to $\inv(\calT')$, and this contradicts the fact that $\phi$ is a bijection (because distinct permutations have distinct inversion sets).
We claim that the set $\inv(\calT\cup \{(q,r)\})$ is transitively closed.

To prove this claim, first suppose that $p<q$ and $(p,q)\in \inv(\calT)$.
We need to show that $(q,r)\in \inv(\calT)$.
Let $M'$ be an indecomposable module in $\calT$ with $\supp(M)= [p+1,q]$.
Then, Lemma~\ref{lem: ext} implies that there exists a short exact sequence:
\[0\to M \to M'' \to M'\to 0,\]
where $\supp(M'')= [p+1,r]$.
Since $M$ is minimally extending, \ref{propertytwo} implies that $M''\in \calT'$.
The statement follows.
A similar argument shows that if there is some $s>r$ such that $(r,s)\in \inv(\calT)$ then $(q,s)$ also belongs to $\inv(\calT)$.
We conclude that $\inv(\calT') = \inv(\calT)\cup \{(q,r)\}$, as desired.
In particular, $\phi(\calT')\covers \phi(\calT)$.

For any finite lattice, define $\covdown(L)$ to be the set of pairs $(w',w)$ such that $w'\covers w$.
As is the case for $\tors RA_n$ and $A_n$, suppose that each element $w'\in L$ has a canonical join represenation.
Recall that the number of canonical joinands of $w'$ is equal to the number of elements covered by $w'$.
(This is Proposition~\ref{prop: labeling}.)
Thus, the number of pairs $(w',w)$ in $\covdown(L)$ is equal to a weighted sum of the faces in $\Gamma(L)$, where each face is weighted by its size.
In particular, the sets $\covdown(\tors RA_n)$ and $\covdown(A_n)$ are equinumerous.
Since $\phi$ is a bijection, we conclude that the mapping $(\calT', \calT) \mapsto (\phi(\calT),\phi(\calT'))$ is a bijection from $\covdown(\tors RA_n)$ to $\covdown(A_n)$.
Hence, the lattice $\tors RA_n$ is isomorphic to the weak order on $A_n$.
\end{proof}

\section{Acknowledgements}
The authors thank David Speyer and Gordana Todorov, who helped start this project.
We also thank Al Garver and Thomas McConville, who led us to the correct statement of Theorem~\ref{main: cjr}.

\label{bib}

\end{document}